\documentclass[a4paper,11pt,onecolumn,twoside]{article}
\usepackage{mathrsfs}
\usepackage{mathrsfs}
\usepackage{amsfonts}
\usepackage{fancyhdr}
\usepackage{amsmath,amsfonts,amssymb,graphicx,amsthm}
\allowdisplaybreaks
\usepackage{amsmath,latexsym,amsfonts,indentfirst,amsthm,amsxtra,amssymb,bm}

\numberwithin{equation}{section}

 \newtheorem{lemma}{Lemma}[section]
 
 \newtheorem{theorem}{Theorem}[section]
 
 \theoremstyle{remark}
 \newtheorem{remark}{Remark}[section]

\begin{document}

\title{\bf Global Existence and Large-time Behavior of Solutions to the Cauchy Problem of One-dimensional Viscous Radiative and Reactive Gas}
\author{{\bf Yongkai Liao}\footnote{Email address: yongkai.liao@whu.edu.cn},\quad {\bf Huijiang Zhao}\thanks{Corresponding author. Email address: hhjjzhao@whu.edu.cn}\\[2mm]
School of Mathematics and Statistics, Wuhan University, Wuhan 430072, China\\
and\\
Computational Science Hubei Key Laboratory, Wuhan University, Wuhan 430072, China}
\date{}

\maketitle

\begin{abstract}
Although there are many results on the global solvability and the precise description of the large time behaviors of solutions to the initial-boundary value problems of the one-dimensional viscous radiative and reactive gas in bounded domain with two typical types of boundary conditions, no result is available up to now for the corresponding problems in unbounded domain.  This paper focuses on the Cauchy problem of such a system with prescribed large initial data and the main purpose is to construct its global smooth non-vacuum solutions around a non-vacuum constant equilibrium state and to study the time-asymptotically nonlinear stability of such an equilibrium state. The key point in the analysis is to deduce the uniform positive lower and upper bounds on the specific volume and the temperature.\vspace{3mm}

\end{abstract}


\section{Introduction}
In this paper we consider a system of equations describing a motion of a one-dimensional gaseous medium in the presence of radiation and reacting process. The model consists of equations corresponding to the conservation laws of the mass, the momentum and the energy coupling with a reaction-diffusion equation which, in the Lagrangian coordinates, can be written as (cf. \cite{Liao-Zhao, Umehara-Tani-JDE-2007, Umehara-Tani-PJA-2008}):
\begin{eqnarray}\label{a1}
    v_t-u_x&=&0,\nonumber\\
    u_t+p\left(v,\theta\right)_x&=&\left(\frac{\mu u_x}{v}\right)_x,\\
    e_t+p(v,\theta)u_x&=&\frac{\mu u_{x}^{2}}{v}+\left(\frac{\kappa\left(v,\theta\right)\theta_{x}}{v}\right)_{x}+\lambda\phi z,\nonumber\\
    z_{t}&=&\left(\frac{dz_x}{v^{2}}\right)_{x}-\phi z.\nonumber
\end{eqnarray}
Here $x\in\mathbb{R}$ is the Lagrangian space variable, $t\in\mathbb{R}^+$ the time variable and the primary dependent variables are the specific volume $v=v\left(t,x\right)$, the velocity\,$u=u\left(t,x\right)$, the absolute temperature\, $\theta=\theta\left(t,x\right)$ and the mass fraction of the reactant\, $z=z\left(t,x\right)$. The positive constants $d$ and $\lambda$ are the species diffusion coefficient and the difference in the heat between the reactant and the product, respectively. The reaction rate function $\phi=\phi\left(\theta\right)$ is defined, from the Arrhenius law \cite{Ducomet-Zlotnik-NonliAnal-2005}, by
\begin{equation}\label{a2}
 \phi\left(\theta\right)=K\theta^{\beta}\exp\left(-\frac{A}{\theta}\right),
\end{equation}
where positive constants $K$ and $A$ are the coefficients of the rate of the reactant and the activation energy, respectively, and $\beta$ is a non-negative number.

We treat the radiation as a continuous field and consider both the wave and photonic effect. Assume that the high-temperature radiation is at thermal equilibrium with the fluid. Then the pressure $p$ and the internal energy $e$ consist of a linear term in $\theta$ corresponding to the perfect polytropic contribution and a fourth-order radiative part due to the Stefan-Boltzmann radiative law \cite{Mihalas-Mihalas-1984}:
\begin{equation}\label{a3}
  p\left(v,\theta\right)=\frac{R\theta}{v}+\frac{a\theta^{4}}{3}, \quad e(v,\theta)=C_{v}\theta+av\theta^{4},
\end{equation}
where the positive constants $R$, $C_{v}$, and $a$ are the perfect gas constant, the specific heat and the Stefan-Boltzmann constant, respectively.

As in \cite{Liao-Zhao, Umehara-Tani-JDE-2007, Umehara-Tani-PJA-2008}, we also assume that the bulk viscosity $\mu$ is a positive constant and the thermal conductivity $\kappa=\kappa\left(v,\theta\right)$ takes the form
\begin{equation}\label{a4}
\kappa\left(v,\theta\right)=\kappa_{1}+\kappa_{2}v\theta^{b}
\end{equation}
with $\kappa_{1}$, $\kappa_{2}$ and $b$ being some positive constants.

The main purpose of this paper is to consider the global solvability and large time behavior of solutions to the Cauchy problem of the system $\eqref{a1}$-\eqref{a4} in $(0,\infty)\times\mathbb{R}$ with prescribed initial data
\begin{equation}\label{a5}
 \left(v\left(0,x\right),u\left(0,x\right),\theta\left(0,x\right), z\left(0,x\right)\right)=\left(v_0\left(x\right),u_0\left(x\right),\theta_{0}\left(x\right), z_{0}\left(x\right)\right)
\end{equation}
for $x\in\mathbb{R}$, which are assumed to satisfy the following far-field condition:
\begin{equation}\label{a6}
\lim_{|x|\rightarrow\infty}\left(v_0\left(x\right),u_0\left(x\right),\theta_{0}\left(x\right), z_{0}\left(x\right)\right)=(1,0,1,0).
 \end{equation}

Before stating our main results, let us review some related results briefly in the literature. The mathematical study of radiation hydrodynamics has attracted a lot of interest in recent years (cf. \cite{Donatelli-Trivisa-CMP-2006, Ducomet-MMAS-1999, Ducomet-ARMA-2004, Ducomet-Feireisl-CMP-2006, Ducomet-Zlotnik-ARMA-2005, Ducomet-Zlotnik-NonliAnal-2005, Jiang-ZHeng-JMP-2012, Jiang-ZHeng-ZAMP-2014, Liao-Zhao, Qin-Hu-Wang-Huang-Ma-JMAA-2013, Umehara-Tani-JDE-2007, Umehara-Tani-PJA-2008} and references cited therein) and, to the best of our knowledge, all these results available up to now focus on the initial-boundary value problem of the system \eqref{a1}-\eqref{a4} in bounded domain with the following two types of boundary conditions:
\begin{itemize}
\item For the case when the initial condition \eqref{a5} is assumed to be hold for $x\in(0,1)$ together with Dirichlet boundary condition for the stress $\sigma=-p(v,\theta)+\frac{\mu u_x}{v}$ and homogeneous Neumann boundary conditions for both $\theta$ and $z$
\begin{equation}\label{a7}
\left(\sigma(t,x), \theta_{x}(t,x), z_{x}(t,x)\right)|_{x=0,1}=\left(-p_{e},0,0\right)\quad {\textrm{for}}\quad t>0
\end{equation}
 with $p_{e}$ being a positive constant, Umehara and Tani \cite{Umehara-Tani-JDE-2007} established global existence, uniqueness of a classical solutions under the assumptions $4\leq b\leq 16$ and $0\leq \beta\leq \frac{13}{2}$. Later on, they improved their results in \cite{Umehara-Tani-PJA-2008} to the case of $b\geq 3$ and $0\leq \beta< b+9$. Moreover, Qin \cite{Qin-Hu-Wang-Huang-Ma-JMAA-2013} further strengthened the results to the case $\left(b, \beta\right)\in E$, where $E=E_{1}\bigcup E_{2}$ with
 \begin{eqnarray*}
E_{1}&&=\left\{\left(b,\beta\right)\in\mathbb{R}^{2}:\quad \frac{9}{4}< b< 3,\ 0\leq \beta< 2b+6\right\},\\
E_{2}&&=\left\{\left(b,\beta\right)\in\mathbb{R}^{2}:\quad 3\leq b,\ 0\leq \beta<b+9\right\}.
 \end{eqnarray*}
Jiang and Zheng \cite{Jiang-ZHeng-JMP-2012} studied global solvability and asymptotic behavior for the problem \eqref{a1}-\eqref{a5}, \eqref{a7} for the case $b\geq 2$ and $0 \leq \beta < b+9$;

\item For the case when the initial condition \eqref{a5} is assumed to be hold for $x\in(0,1)$ together with homogeneous Dirichlet boundary condition for $u$ and homogeneous Neumann boundary conditions for both $\theta$ and $z$
\begin{equation}\label{a8}
\left(u(t,x), \theta_{x}(t,x), z_{x}(t,x)\right)|_{x=0,1}=\left(0,0,0\right)\quad {\textrm{for}}\,\,t>0,
 \end{equation}
global solvability result for the case of $b\geq4$ together with the precise description of the large time behavior of the global solutions constructed above for the case of $b\geq 6$ are obtained by Ducomet in \cite{Ducomet-MMAS-1999}. Later on, Jiang and Zheng \cite{Jiang-ZHeng-ZAMP-2014} improved the result for the case of $b\geq 2$ and $0\leq \beta<b+9$.
\end{itemize}
We note, however, all the results mentioned above are concerned with the case when the space variable $x$ is in a bounded domain. Thus a natural question is : {\it Whether does a similar result on the global solvability and the time-asymptotic nonlinear stability of the non-vacuum equilibrium state $(1,0,1,0)$ hold for the Cauchy problem \eqref{a1}-\eqref{a6} with large initial data or not?} The aim of the present work is devoted to such a problem and the main result can be stated as follows
\begin{theorem}\label{Th1.1}
Suppose that
\begin{itemize}
\item  The parameters $b$ and $\beta$ are assumed to satisfy:
\begin{equation}\label{1.9}
 b>\frac{11}{3}, \quad 0\leq\beta< b+9;
 \end{equation}
\item The initial data $ \left(v_{0}(x), u_{0}(x), \theta_{0}(x), z_{0}(x)\right)$ satisfy
\begin{eqnarray}\label{a9}
  \left(v_{0}(x)-1, u_{0}(x ), \theta_{0}(x)-1, z_{0}(x)\right)\in H^{1}\left(\mathbb{R}\right),\nonumber\\
   u_{0t}(x)\in L^{2}\left(\mathbb{R}\right),\quad z_{0}(x)\in L^{1}\left(\mathbb{R}\right),\\
     \inf\limits_{x\in\mathbb{R}}v_{0}\left(x\right)>0, \quad\inf\limits_{x\in\mathbb{R}} \theta_{0}\left(x\right)>0, \quad 0\leq z_{0}\left(x\right)\leq 1,  \quad \forall x\in\mathbb{R}.\nonumber
  \end{eqnarray}
\end{itemize}
Then the system (\ref{a1})-(\ref{a4}) with prescribed initial data \eqref{a5} satisfying the far-field condition \eqref{a6} admits a unique global solution $\left(v(t,x), u(t,x), \theta(t,x), z(t,x)\right)$ which satisfies
\begin{eqnarray}\label{a11}
\underline{V}\leq v(t,x)&\leq&\overline{V},\nonumber\\
\underline{\Theta}\leq\theta(t,x)&\leq&\overline{\Theta},\\
0\leq z(t,x)&\leq& 1\nonumber
\end{eqnarray}
for all $\left(t,x\right)\in [0,\infty)\times\mathbb{R}$ and
\begin{equation}\label{a12}
\sup\limits_{0\leq t<\infty}\big\|\big(v-1, u, \theta-1, z\big)(t)\big\|_{H^{1}(\mathbb{R})}^{2}+\int_{0}^{\infty}\left(\left\|v_{x}(s)\right\|^{2}_{L^2(\mathbb{R})}+\big\|\big(u_{x}, \theta_{x}, z_{x}\big)(s)\big\|^{2}_{H^{1}\left(\mathbb{R}\right)}\right)ds\leq C.
\end{equation}
Here $\underline{V},$ $\overline{V},$ $\underline{\Theta},$ $\overline{\Theta}$ and $C$ are some positive constants which depend only on the initial data $(v_{0}(x), u_{0}(x), $
$\theta_{0}(x), z_{0}(x))$.

Moreover, the large time behavior of the global solution $\left(v(t,x), u(t,x), \theta(t,x), z(t,x)\right)$ constructed above can be described by the non-vacuum equilibrium state $(1,0,1,0)$ in the sense that
\begin{equation}\label{a13}
\lim_{t\rightarrow+\infty}\left(\left\|\left(v-1, u, \theta-1, z\right)\left(t\right)\right\|_{L^{p}\left(\mathbb{R}\right)}
+\left\|\left(v_{x}, u_{x}, \theta_{x}, z_{x}\right)\left(t\right)\right\|_{L^{2}\left(\mathbb{R}\right)}\right)=0
\end{equation}
holds for any $p\in (2,\infty]$.
\end{theorem}

\begin{remark} Some remarks concerning Theorem \ref{Th1.1} are listed below:
\begin{itemize}
\item Our main result shows that both the specific volume and the temperature are bounded from below and above uniformly in $t$ and $x$. Thus the result for the problem (\ref{a1})-(\ref{a6}) in Theorem \ref{Th1.1} in Lagrangian coordinates can easily be converted to the equivalent statement for the corresponding result for the problem in Eulerian coordinates (cf. \cite{Liao-Zhao, Umehara-Tani-JDE-2007});
\item Note that even for the viscous heat-conducting ideal polytropic gas with constant transport coefficients, although the global solvability result to its Cauchy problem has been established in \cite{Antontsev-Kazhikov-Monakhov-1990, Kazhikhov-Shelukhin-JAMM-1977} for nearly forty years, the problem on its large time behavior is solved only recently in \cite{Li-Liang-ARMA-2016}. Our main result Theorem \ref{Th1.1} shows that similar result holds for certain viscous heat-conducting general gas with constant viscosity and the variable heat-conducting coefficient $\kappa(v,\theta)$ verifying \eqref{a4} provided that the constitutive relations \eqref{a3} are satisfied;
\item The estimate \eqref{a13} obtained in Theorem \ref{Th1.1} tells us that the non-vacuum equilibrium state $(1,0,1,0)$ is time-asymptotically nonlinear stable. For the case when the far fields $(v_\pm, u_\pm, \theta_\pm)$ of the initial data $(v_0(x), u_0(x), \theta_0(x))$ given by \eqref{a5} are different, i.e. $\lim\limits_{x\to\pm\infty} (v_0(x), u_0(x), \theta_0(x))=(v_\pm, u_\pm,\theta_\pm)$ with $(v_-, u_-,\theta_-)\not=(v_+,u_+,\theta_+)$, the precise description of the large time behavior of its global solutions is given by some elementary waves, i.e. rarefaction waves, viscous shock waves, viscous contact waves and/or their suitable superpositions, which are uniquely determined by the structure of the unique global entropy solution of the following Riemann problem of the reduced compressible Euler equations
\begin{eqnarray*}
v_t-u_x&=&0,\quad (t,x)\in \mathbb{R}^+\times\mathbb{R},\\
u_t+p(v,\theta)_x&=&0,\quad (t,x)\in \mathbb{R}^+\times\mathbb{R},\\
\left(e+\frac {u^2}{2}\right)_t+(up(v,\theta))_x&=&0,\quad (t,x)\in \mathbb{R}^+\times\mathbb{R},\\
(v(0,x),u(0,x),\theta(0,x))&=&\left\{
\begin{array}{rl}
(v_-,u_-,\theta_-),& x<0,\\
(v_+,u_+,\theta_+),& x>0.
\end{array}
\right.
\end{eqnarray*}
For such a case, the problem on the large behavior of the global solutions to the Cauchy problem \eqref{a1}-\eqref{a4}, \eqref{a5} with different far-fields can be reduced to the study of the nonlinear stability of above mentioned elementary wave patterns and we're convinced that the argument used in this paper can be adapted to deal with the nonlinear stability of certain wave patterns to the Cauchy problem of \eqref{a1}-\eqref{a4}, \eqref{a5} with large initial data.
\end{itemize}
\end{remark}

Now we outline the main difficulties of the problem and our strategy to deduce our main result obtained in Theorem \ref{Th1.1}. As is usual for the wellposedness theories of nonlinear partial differential equations, the main difficulty in deducing the global solvability results to our problem is to control the possible growth of their solutions induced by the nonlinearities of the equations \eqref{a1} suitably and the key point is to obtain the positive lower and upper bounds on the specific volume $v\left(t,x\right)$ and the temperature $\theta\left(t,x\right)$.

To illustrate our idea, we first recall the arguments used in \cite{Jiang-ZHeng-ZAMP-2014} to deal with the initial-boundary value problem (\ref{a1})-\eqref{a4}, (\ref{a5}), (\ref{a8}). In order to deduce the desired lower and upper bounds of $v(t,x)$, the authors first obtained the following representation formula for the specific volume $v(t,x)$:
\begin{equation}\label{a14}
v\left(t,x\right)=\frac{\widetilde{D}\left(t,x\right)}{\widetilde{B}\left(t\right)}\bigg[1+\frac{1}{\mu}\int_{0}^{t}
\frac{v\left(s,x\right)p\left(s,x\right)\widetilde{B}\left(s\right)}{\widetilde{D}\left(s,x\right)}ds\bigg],
\end{equation}
where
\begin{eqnarray*}
\widetilde{D}\left(t,x\right)&=&v_{0}\left(x\right)\exp\left[\frac{1}{\mu}\left(\int_{0}^{1}v_{0}(x) \left(\int_{0}^{x}u_{0}(y)dy\right)dx+\int_{x_{0}(t)}^{x}u\left(t, y\right)dy-\int_{0}^{x}u_{0}\left(y\right)dy\right)\right],\nonumber\\
\widetilde{B}\left(t\right)&=&\exp\left[\frac{1}{\mu}\int_{0}^{t}\int_{0}^{1}\left(u^{2}+p(v,\theta)v\right)\left(s,x\right)dxds\right].
\end{eqnarray*}
Here $x_{0}(t)\in [0,1]$. Such a method is motivated by an argument developed by Kazhikhov and Shelukhin to study the wellposedness problem of a one-dimensional viscous and heat conducting ideal polytropic gas motion (cf. \cite{Antontsev-Kazhikov-Monakhov-1990, Kazhikhov-Shelukhin-JAMM-1977}). Based on the above explicit representation formula for $v(t,x)$ and the basic energy type estimates, they can derive the desired positive lower and upper bounds on $v(t,x)$ first.

But for the Cauchy problem (\ref{a1})-\eqref{a4}, \eqref{a5}, (\ref{a6}) considered in this paper, since the space variable $x$ is in unbounded domain ($x\in \mathbb{R}$), the above method loses its effect. Furthermore, due to the unboundedness of the domain, we can't obtain a nice bound on $\|v_{x}(t)\|^{2}$ directly as in \cite{Jiang-ZHeng-ZAMP-2014}, which is essential in deriving the upper bond on $\theta\left(t,x\right)$ in \cite{Jiang-ZHeng-ZAMP-2014}. To overcome such difficulties, motivated by the works \cite{Jiang-AMPA-1998, Jiang-CMP-1999, Jiang-PRSE-2002} of Jiang on the viscous heat-conducting ideal polytropic gas, we will derive a new representation formula of $v(t,x)$ by using a special cut-off function to derive the desired bounds on $v(t,x)$. The key point in our analysis can be outlined as in the following:

\begin{itemize}
\item [(i).] We first construct a normalized entropy $\widetilde{S}$ (see (\ref{b14})) to \eqref{a1} to derive the basic energy estimates for our problem, which is essential to deduce the positive lower and upper bounds of the specific volume and temperature. It is worth pointing out that the method to deduce the basic energy estimates here is different from that used in \cite{Umehara-Tani-JDE-2007} due to the unboundedness of the domain;
\item [(ii).] Motivated by \cite{Jiang-CMP-1999}, we use a special cut-off function $\phi\left(x\right)$ (see \eqref{b23}) to derive a new representation of $v\left(t,x\right)$, that is, \eqref{b24}. Based on such a useful formula, we can deduce the desired uniform upper bound of $v\left(t,x\right)$. For the uniform positive lower bound of $v(t,x)$, we can first deduce the uniform lower bound of $v\left(t,x\right)$ for $t\geq t_0$ for some suitably chosen large positive constant $t>0$, then by adopting the method used in \cite{Kawohl-JDE-1985, Kazhikhov-Shelukhin-JAMM-1977} further to yield a lower bound of $v\left(t,x\right)$ for $0< t<t_{0}$, and from these two types of estimates, the desired uniform positive lower bound of $v(t,x)$ follows immedaitely. It is worth pointing out that all the bounds obtained above are independent of the time variable $t$, which is crucial in studying the large-time behavior of our problem;
\item [(iii).] Having obtained the lower and upper bound of $v\left(t,x\right)$, we turn to estimate the term $\|v_{x}(t)\|^{2}$ in terms of $\|\theta\|_{L^\infty([0,T]\times\mathbb{R})}$ in Lemma 2.7, which will be frequently used in deriving the upper bound of $\theta\left(t,x\right)$;
\item [(iv).] Motivated by \cite{Jiang-ZHeng-ZAMP-2014}, we introduce an additional quantity $W(t)$ (see \eqref{b56}) to deduce the upper bound of $\theta\left(t,x\right)$. The estimates obtained here are more delicate than those carried out in \cite{Ducomet-MMAS-1999}. Noticing that the estimates \eqref{b57} and \eqref{bz58} have played an important role in our discussion;
\item [(v).] Finally, by employing the maximum principle, one can derive a local estimate on the lower bound of $\theta\left(t,x\right)$ (see \eqref{b180}). Although such a bound depends on time $t$, it is sufficient to extend the local solution to a global one by combining the above estimates with the continuation argument designed in \cite{Wang-Zhao-M3AS-2016} for the viscous heat-conducting ideal polytropic gas with temperature and density dependent transport coefficients.
\end{itemize}

Note that since the energy producing process inside the medium is taken into account in the equations \eqref{a1}, that is, the gas consists of a reacting mixture and the combustion process is current at the high temperature stage, and the experimental results for gases at high temperatures in \cite{Zeldovich-Raizer-1967} show that the viscosity coefficient $\mu$ may depend on the specific volume and/or temperature. Thus it would be interesting and necessary to consider the corresponding global wellposedness theory for the case when the viscosity coefficient $\mu$ is a function of $v$ and $\theta$.

For such a problem, if the viscosity coefficient $\mu$ is a smooth function of the specific volume $v$ for $v>0$ which can be degenerate, some global solvability results are established in \cite{Liao-Zhao} for the above mentioned two types of initial-boundary value problems of the system \eqref{a1}-\eqref{a4}. As for the case when the viscosity coefficient $\mu$ depends also on the temperature, note that even for one-dimensional compressible Navier-Stokes equations for a viscous and heat conducting ideal polytropic gas, as pointed out in \cite{Jenssen-Karper-SIMA-2010}, temperature dependence of the viscosity $\mu$ has turned out to be especially problematic (for some recent progress in this problem for viscous heat-conducting ideal polytropic gas, see \cite{Huang-Wang-Xiao-KRM-2016, Liu-Yang-Zhao-Zou-SIMA-2014, Wan-Wang-JDE-2017, Wang-Zhao-M3AS-2016}), to the best of our knowledge, no result is available up to now for the system (\ref{a1})-(\ref{a4}) modeling one-dimensional viscous radiative and reactive gas. We're convinced that the argument used in this paper can be used to treat the case when the viscosity coefficient $\mu$ is a smooth function of $v$ and $\theta$ and such a problem is under our current research, cf. \cite{He-Liao-Wang-Zhao-2017}.

Before concluding this section, it is worth pointing out that there are many results on the construction of global smooth non-vacuum solutions to the initial problem and initial-boundary value problems with various boundary conditions for the one-dimensional compressible Navier-Stokes type equations, the interested readers are referred to \cite{Chen-SIMA-1992, Chen-Zhao-Zou-PRSE-2017, Dafermos-Hsiao-NonliAnal-1982, Donatelli-Trivisa-CMP-2006,  Ducomet-M3AS-1996, Ducomet-MMNA-1997, Ducomet-BanachCenterPubl-2000, Ducomet-ARMA-2004, Ducomet-Zlotnik-CRASP-2004, Ducomet-Zlotnik-ARMA-2005, Jenssen-Karper-SIMA-2010, Jiang-ZHeng-JMP-2012, Jiang-ZHeng-ZAMP-2014, Jiang-AMPA-1998, Jiang-CMP-1999, Jiang-PRSE-2002, Kawohl-JDE-1985, Kazhikhov-Shelukhin-JAMM-1977, Li-Liang-ARMA-2016, Liu-Yang-Zhao-Zou-SIMA-2014, Pan-Zhang-CMS-2015, Qin-Hu-JMP-2011, Qin-Hu-Wang-Huang-Ma-JMAA-2013, Qin-Zhang-Su-Cao-JMFM-2016, Shandarin-Zeldovichi-RMP-1989, Tan-Yang-Zhao-Zou-SIMA-2013, Umehara-Tani-JDE-2007, Umehara-Tani-PJA-2008, Wan-Wang-Zhao-JDE-2016, Wan-Wang-Zou-Nonlinearity-2016, Wang-Zhao-M3AS-2016, Wylen-Sonntag-1985, Zeldovich-Raizer-1967, Zhang-Xie-JDE-2008} and the references cited therein.

The rest of the paper is organized as follows: we first derive some useful a priori estimates in Section 2, then the proof of Theorem \ref{Th1.1} will be given in Sections 3.
\bigbreak
\noindent{\bf Notations:}\quad Throughout this paper, $C\geq 1$ is used to denote a generic positive constant which may dependent only on $\inf\limits_{x\in\mathbb{R}}v_{0}\left(x\right)$, $\inf\limits_{x\in\mathbb{R}}\theta_{0}\left(x\right)$, $\|\left(v_{0}-1, u_{0}, \theta_{0}-1, z_{0}\right)\|_{H^{1}{(\mathbb{R}})}$, $\|u_{0t}\|_{L^{2}{(\mathbb{R}})}$ and $\|z_{0}\|_{L^{1}{(\mathbb{R}})}$. Note that these constants may vary from line to line. $C\left(\cdot,\cdot\right)$ stand for some generic positive constant depending only on the quantities listed in the parenthesis.  $\epsilon$ stand for some small positive constant. For function spaces, ~$L^q\left(\mathbb{R}\right)\left(1\leq q\leq \infty\right)$~denotes the usual Lebesgue space on~$\mathbb{R}$~with norm~$\|{\cdot}\|_{L^q\left(\mathbb{R}\right)},$ while for $\ell\in\mathbb{N}$, $H^\ell\left(\mathbb{R}\right)$~denotes the usual $\ell-$th order Sobolev space with norm~$\|{\cdot}\|_\ell\equiv\|{\cdot}\|_{H^\ell\left(\mathbb{R}\right)}$. For simplicity, we use $\|{\cdot}\|_{\infty}$ to denote the norm in $L^{\infty}\left([0,T]\times\mathbb{R}\right)$ for some $T>0$ and use $\|{\cdot}\|$ to denote the norm ~$\|{\cdot}\|_{L^2\left(\mathbb{R}\right)}$.

 \section{Some a priori estimates}
The main purpose of this section is to deduce certain a priori estimates on the solutions of the Cauchy problem \eqref{a1}-\eqref{a4}, \eqref{a5}, \eqref{a6} in terms of the initial data $(v_0(x), u_0(x), \theta_0(x), z_0(x))$. To this end, for some constants $0<T\leq +\infty$, $0<M_1<M_2,$ $0<N_1<N_2$, we first define the set of functions $X(0,T;M_1,M_2;N_1,N_2)$ for which we seek the solution of the Cauchy problem \eqref{a1}-\eqref{a4}, \eqref{a5}, \eqref{a6} as follows:
\begin{eqnarray*}
   &&X(0, T;M_1,M_2;N_1,N_2)\\
&:=&\left\{(v(t,x), u(t,x),\theta(t,x),z(t,x))\ \left|
   \begin{array}{c}
   0\leq z(t,x)\in  C\left(0,T;H^{1}\left(\mathbb{R}\right)\cap L^1(\mathbb{R})\right),\\
   \left(v\left(t,x\right)-1,u\left(t,x\right),\theta\left(t,x\right)-1\right)\in C\left(0,T;H^{1}\left(\mathbb{R}\right)\right),\\
   \left(u_{x}\left(t,x\right), \theta_{x}\left(t,x\right), z_{x}\left(t,x\right)\right)\in L^{2}\left(0,T;H^{1}\left(\mathbb{R}\right)\right),\\
   v_{x}\left(t,x\right)\in L^{2}\left(0,T; L^{2}\left(\mathbb{R}\right)\right),\\
   M_1\leq v(t,x)\leq M_2,\ \forall (t,x)\in[0,T]\times\mathbb{R},\\
   N_1\leq \theta(t,x)\leq N_2,\ \forall (t,x)\in[0,T]\times\mathbb{R}
   \end{array}
   \right.
   \right\}.
  \end{eqnarray*}

The standard local wellposedness result on the Cauchy problem of the hyperbolic-parabolic coupled system tells us that there exists a sufficiently small positive constant $t_1>0$, which depends only on $m_0=\inf\limits_{x\in\mathbb{R}} v_0(x), n_0=\inf\limits_{x\in\mathbb{R}} \theta_0(x)$ and $\ell_0=\|(v_0-1,u_0,\theta_0-1,z_0)\|_1$ such that the Cauchy problem \eqref{a1}-\eqref{a4}, \eqref{a5}, \eqref{a6} admits a unique solution $(v(t,x), u(t,x),\theta(t,x),z(t,x))\in X(0,t_1;m_0/2,2+2\ell_0;n_0/2,2+2\ell_0)$. Now suppose that such a solution has been extended to the time step $t=T\geq t_1$ and $(v(t,x), u(t,x),\theta(t,x),z(t,x))\in X(0,T;M_1,M_2;N_1,N_2)$ for some positive constants $M_2>M_1>0, N_2>N_1>0$, we now try to deduce certain a priori energy type estimates on $(v(t,x), u(t,x),\theta(t,x),z(t,x))$ in terms of the initial data $(v_0(x), u_0(x), \theta_0(x), z_0(x))$.

Our first result is concerned with the basic energy estimates, which will play an essential role in deducing the desired positive lower and upper bounds of $v\left(t,x\right)$. To do so, if we use $S(v,\theta)$ to denote the entropy, then the second law of thermodynamics asserts that
\begin{eqnarray}\label{b4}
   \frac{\partial S\left(v,\theta\right)}{\partial v}&=&\frac{\partial p\left(v,\theta\right)}{\partial\theta},\nonumber\\
    \frac{\partial S\left(v,\theta\right)}{\partial\theta}&=&\frac{1}{\theta}\frac{\partial e\left(v,\theta\right)}{\partial\theta},\\
    \frac{\partial e\left(v,\theta\right)}{\partial v}&=&\theta \frac{\partial p(v,\theta)}{\partial\theta}-p\left(v,\theta\right).\nonumber
\end{eqnarray}
From which and the constitutive relations \eqref{a3}, one easily deduce that
\begin{equation*}\label{b11}
S(v,\theta)=C_{v}\ln\theta+\frac{4}{3}av\theta^{3}+R\ln v
 \end{equation*}
and the normalized entropy $\widetilde{S}(v,\theta)$ around $(v,\theta)=(1,1)$ is given by
\begin{eqnarray}\label{b14}
\widetilde{S}(v,\theta)&&=C_{v}\theta+av\theta^{4}-\left(C_{v}+a\right)+\left(R+\frac{a}{3}\right)\left(v-1\right)-\bigg(S-\frac{4}{3}a\bigg)\nonumber\\
&&=C_{v}\left(\theta-\ln\theta-1\right)+R\left(v-\ln v-1\right)+\frac{1}{3}av\left(\theta-1\right)^{2}\left(3\theta^{2}+2\theta+1\right).
  \end{eqnarray}
Moreover, one can deduce from \eqref{a1}, \eqref{a3} and \eqref{b14} that
\begin{eqnarray}\label{b15}
  &&\left(\widetilde{S}+\frac{u^{2}}{2}\right)_{t}+\frac{\mu u_{x}^{2}}{v\theta}+\frac{\kappa(v,\theta)\theta_{x}^{2}}{v\theta^{2}}+\frac{\lambda\phi z}{\theta}\nonumber\\
  &=&\left\{\left(1-\frac{1}{\theta}\right)\frac{\kappa(v,\theta)\theta_{x}}{v}+\frac{\mu uu_{x}}{v}
  -up(v,\theta)+Ru+\frac{1}{3}au\right\}_{x}+\lambda\phi z.
\end{eqnarray}

Having obtained the above identity, one can easily deduce that
\begin{lemma} [Basic energy estimates] Suppose that $(v(t,x), u(t,x),\theta(t,x),z(t,x))\in X(0,T;M_1,M_2;$ $N_1,N_2)$ for some positive constants $T>0, M_2>M_1>0, N_2>N_1>0$, then for all $0\leq t\leq T$, we have
\begin{equation}\label{b1}
   \int_{\mathbb{R}}z(t,x)dx+\int_{0}^{t}\int_{\mathbb{R}}\phi(t,x) z(t,x)dxds=\int_{\mathbb{R}}z_{0}(x)dx,
\end{equation}
\begin{equation}\label{b2}
 \int_{\mathbb{R}}z^{2}(t,x)dx+2\int_{0}^{t}\int_{\mathbb{R}}\left(\frac{d}{v^{2}}z_{x}^{2}+\phi z^{2}\right)(s)dxds
  =\int_{\mathbb{R}}z_{0}^{2}(x)dx,
\end{equation}
\begin{eqnarray}\label{b3}
&&\int_{\mathbb{R}}\left(\widetilde{S}(t,x)+\frac 12 u^2(t,x)\right)dx
+\int_{0}^{t}\int_{\mathbb{R}}\left(\frac{\mu u_{x}^{2}}{v\theta}+\frac{\kappa(v,\theta)\theta_{x}^{2}}{v\theta^{2}}+\frac{\lambda\phi z}{\theta}\right)(s,x)dxds\nonumber\\
&=&
\int_{\mathbb{R}}\left(\widetilde{S}_0(x)+\frac 12 u_0(x)^2+\lambda z_0(x)\right)dx\\
&\equiv& C_0(v_0,u_0,\theta_0,z_0).\nonumber
\end{eqnarray}
\end{lemma}

The next lemma is concerned with the estimate of $z(t,x)$. To this end, we can deduce by repeating the method used in \cite{Chen-SIMA-1992} that
\begin{lemma} Under the assumptions stated in Lemma 2.1, we have for any $\left(t,x\right)\in [0,T]\times\mathbb{R}$ that
  \begin{equation}\label{b16}
  0\leq z\left(t,x\right)\leq 1.
  \end{equation}
\end{lemma}
The following two lemmas will be useful in deriving bounds on $v\left(t,x\right)$. The first one is concerned with the bounds of $v(t,x)$ and $\theta(t,x)$ at some specially chosen point whose proof can be found in \cite{Jiang-CMP-1999}.
\begin{lemma} Let $a_{1}$, $a_{2}$ be two (positive) roots of the equation $y-\log y-1=\frac{C_0(v_0,u_0,\theta_0,z_0)}{\min \left\{R, C_{v}\right\}}$ with $C_0(v_0,u_0,\theta_0,z_0)$ being given by \eqref{b3}. Then for each $k\in\mathbb{Z}$ and every $0<t\leq T$, there exist $a_{k}\left(t\right), b_{k}\left(t\right)\in \left[-k-1,k+1\right]$ such that
 \begin{equation}\label{b17}
 a_{1}\leq\int_{-k-1}^{k+1}v\left(t,x\right)dx,\quad \int_{-k-1}^{k+1}\theta\left(t,x\right)dx\leq a_{2}
  \end{equation}
and
 \begin{equation}\label{b18}
 a_{1}\leq v\left(t,a_{k}\left(t\right)\right),\quad \theta\left(t,b_{k}\left(t\right)\right)\leq a_{2}
  \end{equation}
hold for $0<t\leq T$.
\end{lemma}
The next lemma is concerned with a rough estimate on $\theta(t,x)$ in terms of the entropy dissipation rate functional
$V(t)=\int_{\mathbb{R}}\left(\frac{\mu u_{x}^{2}}{v\theta}+\frac{\kappa(v,\theta)\theta_{x}^{2}}{v\theta^{2}}\right)(t,x)dx.$ To this end, to simplify the presentation, we set $\Omega_{k}:=\left(-k-1,k+1\right), k\in \mathbb{Z}$ and we can get that
\begin{lemma}For $0\leq m\leq\frac{b+1}{2}$ and each $x\in\mathbb{R}$ (without loss of generality, we can assume that $x\in\Omega_k$ for some $k\in\mathbb{Z}$), we can deduce that
\begin{equation}\label{b19}
\left |\theta^{m}\left(t,x\right)-\theta^{m}\left(t, b_{k}\left(t\right)\right)\right|\leq CV^{\frac{1}{2}}\left(t\right)
\end{equation}
holds for $0\leq t\leq T$ and consequently
\begin{equation}\label{b20}
 \left|\theta\left(t,x\right)\right|^{2m}\leq C+CV\left(t\right), \quad x\in\overline{\Omega}_{k},\ \  0\leq t\leq T.
\end{equation}
\end{lemma}
\begin{proof} For $x\in\Omega_k$, it is easy to see from \eqref{a4} that
\begin{eqnarray}\label{b21}
 \left|\theta^{m}\left(t,x\right)-\theta^{m}\left(t, b_{k}\left(t\right)\right)\right|&&\leq C\int_{\Omega_{k}}\big|\theta^{m-1}\theta_{x}\big|dx\nonumber\\
 &&\leq C\left(\int_{\Omega_{k}}\frac{v\theta^{2m}}{1+v\theta^{b}}dx\right)^{\frac{1}{2}}\left(\int_{\Omega_{k}}\left(\frac{\mu u_{x}^{2}}{v\theta}+
  \frac{\kappa\left(v,\theta\right)\theta_{x}^{2}}{v\theta^{2}}\right)dx\right)^{\frac{1}{2}}.
\end{eqnarray}

Moreover, since
\begin{equation*}
\theta^{2m}\leq C\left(1+\theta^{b+1}\right),
  \end{equation*}
holds for $0\leq m\leq \frac{b+1}{2}$, one thus gets from \eqref{b17} that
\begin{equation}\label{b22}
\int_{\Omega_{k}}\frac{v\theta^{2m}}{1+v\theta^{b}}dx\leq C\int_{\Omega_{k}}\left(v+\theta\right)dx\leq C.
\end{equation}

Combining (\ref{b21}) and (\ref{b22}), we can deduce the estimates \eqref{b19} and \eqref{b20} immediately from \eqref{b18}. This completes the proof of our lemma.
\end{proof}

We now turn to obtain the lower bound and upper bound of the specific volume $v\left(t,x\right)$ which is independent of the time variable $t$. To this end, motivated by the works of Jiang for the one-dimensional viscous and heat-conducting ideal polytropic gas motion (cf. \cite{Jiang-AMPA-1998, Jiang-CMP-1999, Jiang-PRSE-2002}), we first give a local representation of $v\left(t,x\right)$ by using the following cut-off function $\phi\in W^{1,\infty}\left(\mathbb{R}\right)$
\begin{equation}\label{b23}
 \phi\left(x\right)=
 \begin{cases}
 1, & $$x\leq k+1$$,\\
 k+2-x, & $$ k+1\leq x\leq k+2$$,\\
 0, & $$ x\geq k+2$$,
 \end{cases}
\end{equation}
which is the main content of the following lemma.
\begin{lemma} Under the assumptions stated in Lemma 2.1, we have for each $0\leq t\leq T$ that
\begin{equation}\label{b24}
v\left(t,x\right)=B\left(t,x\right)Y\left(t\right)+\frac{1}{\mu}\int_{0}^{t}\frac{B\left(t,x\right)Y\left(t\right)v\left(s,x\right)p\left(s,x\right)}{B\left(s,x\right)Y\left(s\right)}ds,
\quad x\in\overline{\Omega}_k.
\end{equation}
Here
\begin{eqnarray}\label{b25}
B\left(t,x\right)&&:=v_{0}\left(x\right)\exp\left\{\frac{1}{\mu}\int_{x}^{\infty}\left(u_{0}\left(y\right)-u\left(t,y\right)\right)\phi\left(y\right)dy\right\},\nonumber\\
Y\left(t\right)&&:=\exp\left\{\frac{1}{\mu}\int_{0}^{t}\int_{k+1}^{k+2}\sigma\left(s,y\right)dyds\right\},\\
\sigma &&:=-p(v,\theta)+\frac{\mu u_{x}}{v}.\nonumber
\end{eqnarray}
\end{lemma}
\begin{proof}
Multiplying $(\ref{a1})_{2}$ by $\phi$ and integrating the resulting identity over $\left(x,\infty\right)\left(x\in\overline{\Omega}_{k}\right)$ with respect to $x$, we get
\begin{equation}\label{b26}
- \int_{x}^{\infty}\left[u(t,y)\phi(y)\right]_{t}dy=\mu\left[\log v(t,x)\right]_{t}-p(t,x)-\int_{k+1}^{k+2}\sigma\left(t,x\right)dx,\quad x\in\overline{\Omega}_{k}.
\end{equation}

Noticing the definition of $B\left(t,x\right)$ and $Y\left(t\right)$, we integrate (\ref{b26}) over $\left(0,t\right)$ with respect to $t$ and take the exponential on both sides of the resulting equation to deduce that
\begin{equation}\label{b27}
\frac{1}{B\left(t,x\right)Y\left(t\right)}=\frac{1}{v\left(t,x\right)}\exp\left\{\frac{1}{\mu}\int_{0}^{t}p\left(s,x\right)ds\right\},\quad x\in\overline{\Omega}_{k},\  t\geq 0.
\end{equation}

Multiplying (\ref{b27}) by $\frac{p\left(t,x\right)}{\mu}$ and integrating over $(0,t)$, we can infer that
\begin{equation}\label{b28}
\exp\left\{\frac{1}{\mu}\int_{0}^{t}p\left(s,x\right)ds\right\}=1+\frac{1}{\mu}\int_{0}^{t}\frac{v\left(s,x\right)p\left(s,x\right)}{B\left(s,x\right)Y\left(s\right)}ds.
\end{equation}

Combining (\ref{b27}) with (\ref{b28}), we obtain (\ref{b24}). This completes the proof of Lemma 2.5.
\end{proof}

To deduce the desired lower bound of $v\left(t,x\right)$ which is is independent of time $t$, one can first prove by repeating the argument used in \cite{Jiang-CMP-1999} that the following estimates
\begin{equation}\label{b29}
C\left(k\right)^{-1}\leq B\left(t,x\right)\leq C\left(k\right), \quad\forall x\in\overline{\Omega}_{k}
\end{equation}
and
\begin{equation}\label{b30}
-\int_{s}^{t}\inf\limits_{x\in[k+1,k+2]}\theta\left(s,\cdot\right)ds\leq C-\frac{t-s}{C}
\end{equation}
hold true for $0\leq s\leq t\leq T$. Consequently one can get from H\"{o}lder's inequality and Jenssen's inequality $\left(\int_{k+1}^{k+2}vdx\right)^{-1}\leq \int^{k+2}_{k+2}v^{-1}dx$ that
\begin{eqnarray}\label{b31}
&&\int_{s}^{t}\int_{k+1}^{k+2}\left(\frac{\mu u_{x}}{v}-\frac{R\theta}{v}\right)dxds\nonumber\\
&\leq& C\int_{s}^{t}\int_{k+1}^{k+2}\frac{u^{2}_{x}}{v\theta}dxds
-\frac{R}{2}\int_{s}^{t}\int_{k+1}^{k+2}\frac{\theta}{v}dxds\nonumber\\
&\leq& C-\frac{R}{2}\int_{s}^{t}\inf\limits_{x\in[k+1,k+2]}\theta\left(s,\cdot\right)\int_{k+1}^{k+2}\frac{1}{v}dxds\\
&\leq& C-\frac{R}{2}\int_{s}^{t}\inf\limits_{x\in[k+1,k+2]}\theta\left(s,\cdot\right)\bigg(\int_{k+1}^{k+2}vdx\bigg)^{-1}ds\nonumber\\
&\leq& C-\frac{R}{2a_{2}}\int_{s}^{t}\inf\limits_{x\in[k+1,k+2]}\theta\left(s,\cdot\right)ds\nonumber\\
&\leq& C-\frac{t-s}{C},\quad 0\leq s\leq t\leq T.\nonumber
\end{eqnarray}

Using the definition of $Y\left(t\right)$ and (\ref{b31}), we have for $0\leq s\leq t\leq T$ that
\begin{eqnarray}\label{b32}
\frac{Y\left(t\right)}{Y\left(s\right)}&&=\frac{\exp\left\{\frac{1}{\mu}{\displaystyle\int_{s}^{t}\int_{k+1}^{k+2}} \left(\frac{\mu u_{x}}{v}-\frac{R\theta}{v}\right)dxds\right\}}
{\exp\left\{\frac{1}{\mu}{\displaystyle\int_{s}^{t}\int_{k+1}^{k+2}}\frac{a}{3}\theta^{4}dxds\right\}}\nonumber\\
&&\leq\exp\left\{\frac{1}{\mu}\int_{s}^{t}\int_{k+1}^{k+2}\left(\frac{\mu u_{x}}{v}- \frac{R\theta}{v}\right)dxds\right\}\\
&&\leq C\exp\left\{-\frac{t-s}{C}\right\}.\nonumber
 \end{eqnarray}

If we set $s=0$ in (\ref{b32}), we can obtain
\begin{equation}\label{b33}
0\leq Y\left(t\right)\leq C\exp\left\{-\frac{t}{C}\right\}.
 \end{equation}

On the other hand, integrating (\ref{b24}) with respect to $x$ over $\Omega_k=(-k-1,k+1)$ and by using \eqref{b3}, \eqref{b17}, (\ref{b29}) and (\ref{b33}), we can get that
\begin{eqnarray}\label{b34}
a_{1}&&\leq C\int_{-k-1}^{k+1}Y\left(t\right)dx+C\int_{0}^{t}\frac{Y\left(t\right)}{Y\left(s\right)}\int_{-k-1}^{k+1}p\left(s,x\right)v\left(s,x\right)dxds\nonumber\\
&&\leq C\exp\left\{-\frac{t}{C}\right\}+C\int_{0}^{t}\frac{Y\left(t\right)}{Y\left(s\right)}\int_{-k-1}^{k+1} \left(\theta+v\theta^{4}\right)(s,x)dxds\\
&&\leq C\exp\left\{-\frac{t}{C}\right\}+C\int_{0}^{t}\frac{Y\left(t\right)}{Y\left(s\right)}ds.\nonumber
\end{eqnarray}

Furthermore, setting $m=\frac{1}{2}$ in (\ref{b19}) and by using (\ref{b18}), we can deduce that
\begin{eqnarray}
a^{\frac{1}{2}}_{1}-CV^{\frac{1}{2}}\left(t\right)\leq\theta^{\frac{1}{2}}\left(t,x\right),\nonumber
 \end{eqnarray}
which implies
\begin{eqnarray}\label{b35}
\theta\left(t,x\right)\geq\frac{a_{1}}{2}-CV\left(t\right).
 \end{eqnarray}
Thus we can conclude from \eqref{b24}, (\ref{b32}), (\ref{b34}) and (\ref{b35}) that for $x\in\Omega_k, 0\leq t\leq T$
\begin{eqnarray}\label{b36}
v\left(t,x\right)&&\geq\int_{0}^{t}\frac{Y\left(t\right)}{Y\left(s\right)}\theta\left(s,x\right)ds\nonumber\\
&&\geq\int_{0}^{t}\frac{Y\left(t\right)}{Y\left(s\right)}\left(\frac{a_{1}}{2}-CV\left(s\right)\right)ds\\
&&\geq C-C\exp\left\{-\frac{t}{C}\right\}-C\int_{0}^{t}\frac{Y\left(t\right)}{Y\left(s\right)}V\left(s\right)ds,\nonumber
 \end{eqnarray}
while
\begin{eqnarray}\label{b37}
&&\int_{0}^{t}\frac{Y\left(t\right)}{Y\left(s\right)}V\left(s\right)ds\nonumber\\
&\leq& C\int_{0}^{t}\exp\left\{-\frac{\left(t-s\right)}{C}\right\}V\left(s\right)ds\\
&=&C\left(\int_{0}^{\frac{t}{2}}\exp\left\{-\frac{\left(t-s\right)}{C}\right\}V\left(s\right)ds
+\int_{\frac{t}{2}}^{t}\exp\left\{-\frac{\left(t-s\right)}{C}\right\}V\left(s\right)ds\right)\nonumber\\
&\leq& C\left(\exp\left\{-\frac{t}{2C}\right\}\int_{0}^{\frac{t}{2}}V\left(s\right)ds +\int_{\frac{t}{2}}^{t}V\left(s\right)ds\right)\rightarrow 0\quad as\; t\rightarrow +\infty.\nonumber
\end{eqnarray}

Putting (\ref{b36}) and (\ref{b37}) together, we can infer that there exist positive constants $t_{0}$ and $C$ such that if $t\geq t_{0}$, we can get that
\begin{equation}\label{bz38}
v\left(t,x\right)\geq C,\quad \forall t\geq t_0,\ \ x\in\overline{\Omega}_k.
\end{equation}

Having obtained \eqref{bz38}, to deduce a positive lower bound on $v(t,x)$, it suffices to deduce the lower bound of $v\left(t,x\right)$ for $0<t\leq t_{0}$. For this purpose, noticing that $(\ref{a1})_{2}$ can be rewritten as
\begin{eqnarray}\label{bzz38}
u_{t}+p_{x}=\mu\left[\log v\right]_{xt},
 \end{eqnarray}
we can get by integrating (\ref{bzz38}) over $\left(0,t\right)\times\left(a_{k}\left(t\right), x\right)$ with respect to $t$ and $x$ that
\begin{eqnarray}\label{bz39}
-\mu\log v\left(t, x\right)+\int_{0}^{t}p\left(s, x\right)ds&=&\int_{a_{k}\left(t\right)}^{x}\big[u_{0}\left(y\right)-u\left(t,y\right)\big]dy
+\int_{0}^{t}p\left(s, a_{k}\left(t\right)\right)ds\nonumber\\
&&+\mu\log\frac{v_{0}\left(a_{k}\left(t\right)\right)}{v_{0}\left(x\right)v\left(t,a_{k}\left(t\right)\right)}.\nonumber
 \end{eqnarray}

Taking the exponential on both sides of the resulting equation, we obtain
\begin{eqnarray}\label{bz40}
&&\frac{1}{v\left(t,x\right)}\exp\left\{\frac{1}{\mu}\int_{0}^{t}p\left(s, x\right)ds\right\}\nonumber\\
&=&\frac{v_{0}\left(a_{k}\left(t\right)\right)\exp\left\{\frac{1}{\mu}{\displaystyle\int_{a_{k}(t)}^{x}}\big[u_{0}\left(y\right) -u\left(t,y\right)\big]dy\right\}}
{v\left(t,a_{k}\left(t\right)\right)v_{0}\left(x\right)}
\exp\left\{\frac{1}{\mu}{\displaystyle\int_{0}^{t}}p\left(s,a_{k}\left(t\right)\right)ds\right\}\\
&:=&B_{k}\left(t,x\right)Y_{k}\left(t\right).\nonumber
\end{eqnarray}

Obviously, we have
\begin{equation}\label{bz41}
C^{-1}\left(k\right)\leq B_{k}\left(t,x\right)\leq C\left(k\right).
\end{equation}

Multiplying (\ref{bz40}) by $\frac{p\left(t,x\right)}{\mu}$ and integrating over $(0,t)$, we can infer
\begin{equation}\label{bz42}
\exp\left\{\frac{1}{\mu}\int_{0}^{t}p\left(s,x\right)ds\right\}=1+\frac{1}{\mu}\int_{0}^{t}v\left(s,x\right)p\left(s,x\right)B_{k}\left(s,x\right)Y_{k}\left(s\right)ds.
\end{equation}

Combining (\ref{bz40}) with (\ref{bz42}), we arrive at
\begin{equation}\label{bz43}
v\left(t,x\right)Y_{k}\left(t\right)=B^{-1}_{k}\left(t,x\right)\left(1 +\frac{1}{\mu}\int_{0}^{t}v\left(s,x\right)p\left(s,x\right)B_{k}\left(s,x\right)Y_{k}\left(s\right)ds\right).
\end{equation}

Integrating (\ref{bz43}) with respect to $x$ over $\left(-k-1,k+1\right)$ and by using \eqref{b3}, \eqref{b17} and (\ref{bz41}), one has
\begin{equation}\label{bz44}
Y_{k}\left(t\right)\leq C\left(1+\int_{0}^{t}\left(\int_{-k-1}^{k+1}v\left(s,x\right)p\left(s,x\right)dx\right)Y_{k}\left(s\right)ds\right).\nonumber
\end{equation}

With the help of Gronwall's inequality, we have for $0\leq t\leq t_0$ that
\begin{eqnarray}\label{bz45}
Y_{k}\left(t\right)&&\leq C\exp\left\{\int_{0}^{t}\int_{-k-1}^{k+1}v\left(s,x\right)p\left(s,x\right)dxds\right\}\nonumber\\
&&\leq C\exp\left\{\int_{0}^{t}\int_{-k-1}^{k+1}\left(R\theta+\frac{av\theta^{4}}{3}\right)\left(s, x\right)dxds\right\}\\
&&\leq C\exp\left\{Ct_{0}\right\}.\nonumber
\end{eqnarray}

Combining (\ref{bz40}), (\ref{bz41}) and (\ref{bz45}), we can deduce that there exists a positive constant $C$ such that the following estimate
\begin{eqnarray}\label{b38}
v\left(t,x\right)\geq C
\end{eqnarray}
holds for $0\leq t\leq t_0$ and $x\in\overline{\Omega}_k$.

For the upper bound of $v\left(t, x\right)$, due to
\begin{eqnarray}\label{b39}
v\left(t,x\right)&&\leq CY\left(t\right)+C\int_{0}^{t}\frac{Y\left(t\right)}{Y\left(s\right)}v\left(s, x\right)p\left(s, x\right)ds\nonumber\\
&&\leq C+C\int_{0}^{t}\exp\left\{-\frac{t-s}{C}\right\}\left(\theta+v\theta^{4}\right)\left(s, x\right)ds,
 \end{eqnarray}
and noticing that the following two estimates
\begin{eqnarray}\label{b40}
\int_{0}^{t}\exp\left\{-\frac{\left(t-s\right)}{C}\right\}\theta\left(s, x\right)ds&&\leq C\int_{0}^{t}\exp\left\{-\frac{\left(t-s\right)}{C}\right\}(1+V\left(s\right))ds\nonumber\\
&&\leq C
\end{eqnarray}
and
\begin{equation}\label{b41}
\int_{0}^{t}\exp\left\{-\frac{\left(t-s\right)}{C}\right\}v(s,x)\theta^{4}\left(s, x\right)ds\leq C\int_{0}^{t}v\left(s,x\right)\exp\left\{-\frac{\left(t-s\right)}{C}\right\}(1+V\left(s\right))ds
\end{equation}
hold for $x\in\overline{\Omega}_{k}, t\geq 0$, we can get from (\ref{b39})-(\ref{b41}) and by using Gronwall's inequality that
\begin{equation}\label{b42}
v\left(t,x\right)\leq C\left(k\right)
\end{equation}
holds for some positive constant $C(k)$. Here it is worth to point out that we have used Lemma 2.4 with $m=\frac{1}{2}$ and $m=2$ in deducing \eqref{b41}.

So far, for each $|x|\leq k$, we have deduced the uniform positive lower bound and the upper bound of $v\left(t,x\right)$ which depends only on $k$ and the initial data but is independent of the time variable $t$. Such an estimate together with the fact $v\left(t,x\right)-1\in C\left(0,T;H^{1}\left(\mathbb{R}\right)\right)$ tell us that
\begin{lemma} Under the assumptions listed in Lemma 2.1, there exist positive constants $\underline{V}$, $\overline{V}$, which depend only on the initial data $(v_0(x), u_0(x), \theta_0(x), z_0(x))$, such that
\begin{equation}\label{b43}
 \underline{V}\leq v\left(t,x\right)\leq \overline{V}
\end{equation}
holds for all $(t,x)\in[0,T]\times\mathbb{R}$.
\end{lemma}

To deduce the desired bounds on $\theta(t,x)$, we first deduce a bound on $\|v_{x}(t)\|$ in terms of $\|\theta\|_{\infty}$, which will be used later on. For this purpose, due to
\begin{equation*}
\frac{u_tv_{x}}{v}=\left(\frac{uv_{x}}{v}\right)_{t}-\left(\frac{uu_{x}}{v}\right)_{x}+\frac{u^{2}_{x}}{v}
\end{equation*}
and noticing
\begin{equation*}
\mu\left(\frac{v_x}{v}\right)_t=u_t+p(v,\theta)_x,
\end{equation*}
we can get by multiplying the last equation by $\frac{v_{x}}{v}$ and integrating the resulting identity with respect to $t$ and $x$ over $\left(0,t\right)\times\mathbb{R}$ that
\begin{eqnarray}\label{b45}
&&\frac{\mu}{2}\int_{\mathbb{R}}\frac{v_{x}^{2}}{v^{2}}dx+R\int_{0}^{t}\int_{\mathbb{R}}\frac{\theta v_{x}^{2}}{v^{3}}dxds\nonumber\\
&=&\int_\mathbb{R}\left(\frac{\mu}{2}\frac{v_{0x}^2}{v_0^2}-\frac{u_0v_{0x}}{v_0}\right)dx+\underbrace{\int_\mathbb{R}\frac{uv_x}{v}dx}_{I_1}
+\underbrace{\int^t_0\int_\mathbb{R}\frac{u_x^2}{v}dxds}_{I_2}\\
&&+R\underbrace{\int^t_0\int_\mathbb{R}\frac{v_x\theta_x}{v^2}dxds}_{I_3}
+\frac{4a}{3}\underbrace{\int^t_0\int_\mathbb{R}\frac{\theta^3v_x\theta_x}{v}dxds}_{I_4}.\nonumber
\end{eqnarray}

Now we deal with $I_k (1\leq k\leq 4)$ term by term. To this end, we first get by employing Cauchy's inequality, Sobolev's inequality and Lemma 2.1 that
\begin{eqnarray}\label{b46}
I_{1}&&\leq\epsilon\int_{\mathbb{R}} \frac{v_{x}^{2}}{v^{2}}dx
+C\left(\epsilon\right)\int_{\mathbb{R}}u^{2}dx\\
&&\leq\epsilon\int_{\mathbb{R}} \frac{v_{x}^{2}}{v^{2}}dx
+C\left(\epsilon\right),\nonumber
\end{eqnarray}

\begin{eqnarray}\label{b47}
I_{2}&&\leq\epsilon\int_{0}^{t}\int_{\mathbb{R}}\frac{\theta v_{x}^{2}}{v^{3}}dxds
+C\left(\epsilon\right)\int_{0}^{t}\int_{\mathbb{R}}\frac{\kappa(v,\theta)\theta^{2}_{x}}{v\theta^{2}}\cdot\frac{\theta}{\kappa(v,\theta)}dxds\nonumber\\
&&\leq\epsilon\int_{0}^{t}\int_{\mathbb{R}}\frac{\theta v_{x}^{2}}{v^{3}}dxds
+C\left(\epsilon\right)
\end{eqnarray}
and
\begin{equation}\label{b48}
I_{3}=\left|\int_{0}^{t}\int_{\mathbb{R}}\frac{u^{2}_{x}}{v\theta}\cdot\theta dxds\right|\leq C\|\theta\|_{\infty}.
\end{equation}

As to $I_4$, due to
\begin{eqnarray}\label{b49}
I_{4}&&\leq\epsilon\int_{0}^{t}\int_{\mathbb{R}}\frac{\theta v_{x}^{2}}{v^{3}}dxds
+C\left(\epsilon\right)\int_{0}^{t}\int_{\mathbb{R}}\frac{\kappa(v,\theta)\theta^{2}_{x}} {v\theta^{2}}\cdot\frac{\theta^{7}}{\kappa(v,\theta)}dxds\nonumber\\
&&\leq\epsilon\int_{0}^{t}\int_{\mathbb{R}}\frac{\theta v_{x}^{2}}{v^{3}}dxds
+C\left(\epsilon\right)\int_{0}^{t}\int_{\mathbb{R}}\frac{\kappa(v,\theta)\theta^{2}_{x}}{v\theta^{2}}\cdot\frac{\theta^{7}}{1+\theta^{b}}dxds,
\end{eqnarray}
one can easily deduce that
\begin{equation}\label{b50}
I_{4}\leq\epsilon\int_{0}^{t}\int_{\mathbb{R}}\frac{\theta v_{x}^{2}}{v^{3}}dxds
+C\left(\epsilon\right)
\end{equation}
if $b\geq 7$, while if $b<7$, one can deduce that
\begin{equation}\label{b52}
I_{4}\leq\epsilon\int_{0}^{t}\int_{\mathbb{R}}\frac{\theta v_{x}^{2}}{v^{3}}dxds
+C\left(\epsilon\right)\|\theta\|^{7-b}_{\infty}.
\end{equation}
Thus one gets from \eqref{b50} and \eqref{b52} that
\begin{equation}\label{b54}
I_{4}\leq\epsilon\int_{0}^{t}\int_{\mathbb{R}}\frac{\theta v_{x}^{2}}{v^{3}}dxds
+C\left(\epsilon\right)\|\theta\|^{(7-b)_{+}}_{\infty}.
\end{equation}
Here $(7-b)_{+}:=\max\{0, 7-b\}$.

Inserting the above estimates on $I_k (k=1,2,3,4)$ into \eqref{b45}, we can get from the estimate \eqref{b43} obtained in Lemma 2.6 that
\begin{lemma} Under the assumptions listed in Lemma 2.1, we have for any $0\leq t\leq T$ that
\begin{equation}\label{bz57}
\left\|v_{x}(t)\right\|^2+\int_{0}^{t}\left\|\sqrt{\theta(s)} v_{x}(s)\right\|^2 ds\leq C+C\|\theta\|^{\max\{1,(7-b)_{+}\}}_{\infty}.
\end{equation}
\end{lemma}

Now we turn to derive an estimate on the upper bound on $\theta\left(t,x\right)$. For this purpose, we set
\begin{eqnarray}\label{b56}
X(t):&=&\int_{0}^{t}\int_{\mathbb{R}}\left(1+\theta^{b+3}(s,x)\right)\theta_{t}^{2}(s,x)dxds,\nonumber\\ Y(t):&=&\max\limits_{s\in(0,t)}\int_{\mathbb{R}}\left(1+\theta^{2b}(s,x)\right)\theta_{x}^{2}(s,x)dx,\\
Z(t):&=&\max\limits_{s\in(0,t)}\int_{\mathbb{R}}u_{xx}^{2}(s,x)dx,\nonumber\\
W(t):&=&\int_{0}^{t}\int_{\mathbb{R}}u^{2}_{xt}(s,x)dxds\nonumber
\end{eqnarray}
and then try to deduce certain estimates between them by employing the structure of the system \eqref{a1}, \eqref{a2}, \eqref{a3} and \eqref{a4} under our consideration.

Firstly, for each $x\in\mathbb{R}$, there exists an integer $k\in\textbf{Z}$ such that $x\in \left[-k-1, k+1\right]$ and we can assume without loss of generality that $x\geq b_{k}\left(t\right)$. Observe first that
\begin{eqnarray}\label{b57}
 \left(\theta(t,x)-1\right)^{2b+6}&&=\left(\theta\left(t,b_{k}(t)\right)-1\right)^{2b+6}
 +\int_{b_{k}\left(t\right)}^{x}\left(2b+6\right)\left(\theta(t,y)-1\right)^{2b+5}\theta_{x}(t,y)dy\nonumber\\
 &&\leq C+C\int_{\mathbb{R}}\left|\theta(t,x)-1\right|^{2b+5}|\theta_{x}(t,x)|dx \\
 &&\leq C+C\|\theta(t)-1\|_{L^{\infty}\left(\mathbb{R}\right)}^{b+3}\left(\int_{\mathbb{R}}\left(\theta(t,x)-1\right)^{4} dx\right)^{\frac{1}{2}}\left(\int_{\mathbb{R}}\left(\theta(t,x)-1\right)^{2b}\theta_{x}^{2}(t,x)dx\right)^{\frac{1}{2}} \nonumber\\
 &&\leq C+C\|\theta(t)-1\|^{b+3}_{L^{\infty}\left(\mathbb{R}\right)}Y^{\frac{1}{2}}(t),\nonumber
\end{eqnarray}
which implies
\begin{equation}\label{b58}
\|\theta(t)\|_{L^{\infty}\left(\mathbb{R}\right)} \leq C+CY(t)^{\frac{1}{2b+6}},
\end{equation}
where we have used the fact that
\begin{equation}\label{bz58}
\left(\theta-1\right)^{4}\leq \left(\theta-1\right)^{2}\left(3\theta^{2}+2\theta+1\right),\quad\left(\theta-1\right)^{2b}\leq C\left(1+\theta^{2b}\right).
\end{equation}

Secondly, by the Gagliardo-Nirenberg inequality, we infer that
\begin{equation*}
\|u_{x}(t)\|\leq C\|u(t)\|^{\frac{1}{2}}\|u_{xx}(t)\|^{\frac{1}{2}}\leq C\|u_{xx}(t)\|^{\frac{1}{2}},
\end{equation*}
which implies that
\begin{equation}\label{b60}
\max\limits_{s\in(0,t)}\int_{\mathbb{R}}u_{x}^{2}(s,x)dx\leq C+CZ(t)^{\frac{1}{2}}.
\end{equation}

Furthermore, by the Sobolev inequality, we can get that
\begin{eqnarray}\label{b61}
\|u_{x}(t)\|_{L^{\infty}\left(\mathbb{R}\right)}&&\leq C\|u_{x}(t)\|^{\frac{1}{2}}\|u_{xx}(t)\|^{\frac{1}{2}}\nonumber\\
&&\leq C\left(1+Z(t)^{\frac{1}{8}}\right)Z(t)^{\frac{1}{4}}\\
&&\leq C+CZ(t)^{\frac{3}{8}}.\nonumber
  \end{eqnarray}

With the above preparations in hand, our next result is to show that $X(t)$ and $Y(t)$ can be controlled by $Z(t)$ and $W(t)$.
\begin{lemma} Under the assumptions listed in Lemma 2.1, we have for $0\leq t\leq T$ that
\begin{equation}\label{b62}
X(t)+Y(t)\leq C\left(1+Z(t)^{\lambda_{1}}\right)+\epsilon W(t).
\end{equation}
Here $\epsilon>0$ canbe chosen as small as we wanted and $\lambda_1$ is given by
\begin{equation}\label{bz136}
\lambda_{1}=\max\left\{\frac{b+3}{b+5}, \frac{3(b+3)}{8(b+2)}, \frac{1}{2},\frac{3(b+3)}{2(3b+9-2\max\{1,\; (7-b)_{+}\})}\right\}.
\end{equation}
It is easy to see that $\lambda_1\in(0,1)$ when $b>\frac{19}{7}$.
\end{lemma}
\begin{proof}
In the same manner as in \cite{Kawohl-JDE-1985, Umehara-Tani-JDE-2007}, if we set
\begin{equation}\label{b63}
K\left(v,\theta\right)=\int_{0}^{\theta}\frac{\kappa\left(v,\xi\right)}{v}d\xi=\frac{\kappa_1\theta}{v}+\frac{\kappa_2\theta^{b+1}}{b+1},
  \end{equation}
then it is easy to verify from the estimate \eqref{b43} obtained in Lemma 2.6 that
\begin{eqnarray}
K_{t}(v,\theta)&=&K_{v}(v,\theta)u_{x}+\frac{\kappa(v,\theta)\theta_{t}}{v},\label{b64}\\
K_{xt}(v,\theta)&=&\left[\frac{\kappa(v,\theta)\theta_{x}}{v}\right]_{t}+K_{v}(v,\theta)u_{xx}+K_{vv}(v,\theta)v_{x}u_{x}
+\left(\frac{\kappa(v,\theta)}{v}\right)_{v}v_{x}\theta_{t},\label{b65}\\
\left|K_{v}(v,\theta)\right|&+&\left|K_{vv}(v,\theta)\right|\leq C\theta.\label{b66}
\end{eqnarray}

Multiplying $(\ref{a1})_{3}$ by $K_{t}$ and integrating the resulting identity over $\left(0,t\right)\times\mathbb{R}$, we arrive at
\begin{eqnarray}\label{b67}
&&\int_{0}^{t}\int_{\mathbb{R}}\left(e_{\theta}(v,\theta)\theta_{t}+\theta p_{\theta}(v,\theta)u_{x}-\frac{\mu u_{x}^{2}}{v}\right)K_{t}(v,\theta)dxd\tau\nonumber\\ &&+\int_{0}^{t}\int_{\mathbb{R}}\frac{\kappa\left(v,\theta\right)}{v}\theta_{x}K_{tx}(v,\theta)dxds\\
&=&\int_{0}^{t}\int_{\mathbb{R}}\lambda\phi zK_{t}(v,\theta)dxds.\nonumber
\end{eqnarray}

Combining (\ref{b63})-(\ref{b67}), we have
\begin{eqnarray}\label{b68}
&&\int_{0}^{t}\int_{\mathbb{R}}\frac{e_{\theta}(v,\theta)\kappa\left(v,\theta\right)\theta_{t}^{2}}{v}dxds
+\int_{0}^{t}\int_{\mathbb{R}}\frac{\kappa\left(v,\theta\right)\theta_{x}}{v}\left(\frac{\kappa\left(v,\theta\right)\theta_{x}}{v}\right)_{t}dxds\nonumber\\
&\leq& C+\underbrace{\left|\int_{0}^{t}\int_{\mathbb{R}}e_{\theta}(v,\theta)\theta_{t}K_{v}(v,\theta)u_{x}dxds\right|}_{I_5}
+\underbrace{\left|\int_{0}^{t}\int_{\mathbb{R}}\theta p_{\theta}(v,\theta)u_{x}K_{v}(v,\theta)u_{x}dxds\right|}_{I_6}\nonumber\\
&&+\underbrace{\left|\int_{0}^{t}\int_{\mathbb{R}}\frac{\theta p_{\theta}(v,\theta)\kappa\left(v,\theta\right)u_{x}\theta_{t}}{v}dxds\right|}_{I_7}
+\underbrace{\left|\int_{0}^{t}\int_{\mathbb{R}}\frac{\mu u_{x}^{2}K_{t}(v,\theta)}{v}dxds\right|}_{I_8}\\
&&+\underbrace{\left|\int_{0}^{t}\int_{\mathbb{R}}\frac{\kappa(v,\theta)}{v}\theta_{x}\left(K_{vv}(v,\theta)v_{x}u_{x} +K_{v}(v,\theta)u_{xx}\right)dxds\right|}_{I_9}\nonumber\\
&&+\underbrace{\left|\int_{0}^{t}\int_{\mathbb{R}}\frac{\kappa(v,\theta)\theta_{x}}{v} \left(\frac{\kappa(v,\theta)}{v}\right)_{v}v_{x}\theta_{t}dxds\right|}_{I_{10}}\nonumber\\
&&+\underbrace{\left|\int_{0}^{t}\int_{\mathbb{R}}\lambda\phi zK_{v}(v,\theta)u_{x}dxds\right|}_{I_{11}}
+\underbrace{\left|\int_{0}^{t}\int_{\mathbb{R}}\frac{\lambda\phi z\kappa\left(v,\theta\right)\theta_{t}}{v}dxds\right| }_{I_{12}}.\nonumber
\end{eqnarray}

We now turn to control $I_k (k=5,6,\cdots, 12)$ term by term. To do so, we first have from \eqref{a3}, \eqref{a4} and \eqref{b43} that
\begin{eqnarray}\label{b69}
\int_{0}^{t}\int_{\mathbb{R}}\frac{e_{\theta}(v,\theta)\kappa\left(v,\theta\right)\theta_{t}^{2}}{v}dxds&&\geq C\int_{0}^{t}\int_{\mathbb{R}}\left(1+\theta^{3}\right)\left(1+\theta^{b}\right)\theta_{t}^{2}dxds \nonumber\\
&&\geq CX(t)
\end{eqnarray}
and
\begin{eqnarray}\label{b70}
&&\int_{0}^{t}\int_{\mathbb{R}}\frac{\kappa\left(v,\theta\right)\theta_{x}}{v}\left(\frac{\kappa\left(v,\theta\right)\theta_{x}}{v}\right)_{t}dxds\nonumber\\
&=&\frac{1}{2}\int_{\mathbb{R}}
\left(\frac{\kappa\left(v,\theta\right)\theta_{x}}{v}\right)^{2}(t,x)dx-\frac{1}{2}\int_{\mathbb{R}}\left(\frac{\kappa\left(v,\theta\right)\theta_{x}}{v}\right)^{2}\left(0,x\right)dx\\
&\geq& CY(t)-C.\nonumber
\end{eqnarray}

With the above two estimates in hand, $I_k (k=5,6, 7, 8)$ can be estimated term by term by employing Cauchy's inequality, H\"{o}lder's inequality and Young's inequality as follows:
\begin{eqnarray}\label{b105}
|I_{5}|&&\leq C\int_{0}^{t}\int_{\mathbb{R}} \left(1+\theta\right)^{4}\left|\theta_{t}u_{x}\right|dxds \nonumber\\
&&\leq \epsilon X(t) +C\left(\epsilon\right)\left(1+\|\theta\|_{\infty}^{(6-b)_{+}}\right)\int_{0}^{t}\int_{\mathbb{R}}\frac{u_{x}^{2}}{\theta}dxds\\
&&\leq \epsilon X(t)+C\left(\epsilon\right)\left(1+Y(t)^{\frac{(6-b)_{+}}{2b+6}}\right)\nonumber\\
&&\leq \epsilon (X(t)+Y(t))+C\left(\epsilon\right),\nonumber
\end{eqnarray}
\begin{eqnarray}\label{b106}
|I_{6}|&&\leq C\int_{0}^{t}\int_{\mathbb{R}}\left(1+\theta\right)^{5}u_{x}^{2}dxds\nonumber\\
&&\leq C\int_{0}^{t}\int_{\mathbb{R}}\bigg(\frac{u_{x}^{2}}{\theta}\cdot\theta+\frac{u_{x}^{2}}{\theta}\cdot\theta^{6}\bigg)dxds\\
&&\leq C+C\|\theta\|^{6}_{\infty}\nonumber\\
&&\leq C\left(\epsilon\right)+\epsilon Y(t),\nonumber
\end{eqnarray}
where we have used (\ref{b3}) and (\ref{b58}),
\begin{eqnarray}\label{b107}
|I_{7}|&&\leq C\int_{0}^{t}\int_{\mathbb{R}}\left(1+\theta\right)^{b+4}\left|u_{x}\theta_{t}\right|dxds\nonumber\\
&&\leq \epsilon X(t)+C\left(\epsilon\right) \left(1+\|\theta\|^{b+6}_{\infty}\right)\int_{0}^{t}\int_{\mathbb{R}}\frac{u_{x}^{2}}{\theta}dxds\\
&&\leq \epsilon X(t)+CY(t)^{\frac{b+6}{2b+6}} \nonumber\\
&&\leq \epsilon\left(X(t)+Y(t)\right)+C\left(\epsilon\right),\nonumber
\end{eqnarray}
and
\begin{eqnarray}\label{b108}
|I_{8}|&=&\left|\left.\int_{\mathbb{R}}\frac{\mu u_{x}^{2}K(v,\theta)}{v}dx\right|_{s=0}^{s=t}-\int_{0}^{t}\int_{\mathbb{R}}\left(\frac{\mu u_{x}^{2}}{v}\right)_{t}K(v,\theta)dxds\right|\nonumber\\
&\leq&C+\left|\int_{\mathbb{R}}\frac{\mu u_{x}^{2}K(v,\theta)}{v}dx\right| +\left|\int_{0}^{t}\int_{\mathbb{R}}\left(\frac{2\mu u_{x}u_{xt}}{v}-\frac{\mu u^{3}_{x}}{v^{2}}\right)K(v,\theta)dxds\right|\\
&\leq&\epsilon(Y(t)+W(t))+C\left(\epsilon\right)\left(1+Z(t)^{\frac{b+3}{b+5}}\right).\nonumber
\end{eqnarray}
Here we have the following two estimates:
\begin{eqnarray}\label{b109}
\left|\int_{\mathbb{R}}\frac{\mu u_{x}^{2}K(v,\theta)}{v}dx\right|&&\leq C\left(1+\left\|\theta\right\|_\infty^{b+1}\right)\max\limits_{0\leq s\leq t}\|u_{x}(s)\|^{2}\nonumber\\
&&\leq C\left(1+Y(t)^{\frac{b+1}{2b+6}}\right)\left(1+Z(t)^{\frac{1}{2}}\right)\\
&&\leq \epsilon Y(t)+C\left(\epsilon\right)Z(t)^{\frac{b+3}{b+5}}\nonumber
\end{eqnarray}
and
\begin{eqnarray}\label{b110}
&&\left|\int_{0}^{t}\int_{\mathbb{R}}\left(\frac{2\mu u_{x}u_{xt}}{v}-\frac{\mu u^{3}_{x}}{v^{2}}\right)K(v,\theta)dxds\right|\nonumber\\
&\leq& C\int_{0}^{t}\int_{\mathbb{R}}(1+\theta)^{b+1}\left(|u_{xt}u_{x}|+|u_{x}^{3}|\right)dxds\nonumber\\
&\leq& C\left(1+\|\theta\|^{\frac{2b+3}{2}}_{\infty}\right)\left(\int_{0}^{t}\|u_{xt}(s)\|^{2}ds\right)^{\frac{1}{2}}
\left(\int_{0}^{t}\int_{\mathbb{R}}\frac{u^{2}_{x}}{\theta}dxds\right)^{\frac{1}{2}}\\
&&+C\left(1+\|\theta\|^{b+2}_{\infty}\right)\|u_{x}\|_{\infty}\int_{0}^{t}\int_{\mathbb{R}}\frac{u^{2}_{x}}{\theta}dxds\nonumber\\
&\leq& C\left(1+Y(t)^{\frac{2b+3}{4b+12}}\right)W(t)^{\frac{1}{2}} +C\left(1+Y(t)^{\frac{b+2}{2b+6}}\right)\left(1+Z(t)^{\frac{3}{8}}\right)\nonumber\\
&\leq& \epsilon(Y(t)+W(t))+C\left(\epsilon\right)\left(1+Z(t)^{\frac{3b+9}{4b+16}}\right).\nonumber
\end{eqnarray}
Here we have used the following facts
\begin{eqnarray*}
Y(t)^{\frac{b+1}{2b+6}}Z(t)^{\frac{1}{2}}&&\leq\epsilon Y(t)+C\left(\epsilon\right)Z(t)^{\frac{b+3}{b+5}},\\
Y(t)^{\frac{b+2}{2b+6}}Z(t)^{\frac{3}{8}}&&\leq\epsilon Y(t)+C\left(\epsilon\right)Z(t)^{\frac{3b+9}{4b+16}},\\
\left(1+Y(t)^{\frac{2b+3}{4b+12}}\right)W(t)^{\frac{1}{2}}&&\leq\epsilon W(t)+C\left(\epsilon\right)\left(1+Y(t)^{\frac{2b+3}{2b+6}}\right)\\
&&\leq\epsilon(Y(t)+W(t))+C\left(\epsilon\right).
\end{eqnarray*}

For $I_9$, noticing that
\begin{eqnarray}\label{b112}
&&\int_{0}^{t}\int_{\mathbb{R}}\frac{\kappa(v,\theta)}{v}\theta_{x}K_{v}(v,\theta)u_{xx}dxds\nonumber\\
&=&-\int_{0}^{t}\int_{\mathbb{R}}\left(\frac{\kappa(v,\theta)\theta_{x}}{v}\right)_{x}K_{v}(v,\theta)u_{x}dxds\\
&&-\int_{0}^{t}\int_{\mathbb{R}}\frac{\kappa(v,\theta)}{v}\theta_{x}u_{x}\left(K_{vv}(v,\theta)v_{x} +\left(\frac{\kappa(v,\theta)}{v}\right)_{v}\theta_{x}\right)dxds,\nonumber
\end{eqnarray}
thus
\begin{eqnarray}\label{b113}
I_{9}&\leq&\underbrace{\left|\int_{0}^{t}\int_{\mathbb{R}}\frac{\kappa(v,\theta)}{v}\theta_{x}^{2} \left(\frac{\kappa(v,\theta)}{v}\right)_{v}u_{x}dxds\right|}_{I^1_9}\nonumber\\
&&+\underbrace{\left|\int_{0}^{t}\int_{\mathbb{R}}\left(\frac{\kappa(v,\theta) \theta_{x}}{v}\right)_{x}K_{v}(v,\theta)u_{x}dxds\right|}_{I_9^2}.
\end{eqnarray}

It is easy to see that $I_9^1$ can be estimated as
\begin{eqnarray}\label{b114}
I_9^1&&\leq C\left(1+\|\theta\|^{2}_{\infty}\right)\|u_{x}\|_{\infty}\int_{0}^{t} \int_{\mathbb{R}}\frac{\kappa(v,\theta)\theta_{x}^{2}}{\theta^{2}}dxds\nonumber\\
&&\leq C\left(1+Y(t)^{\frac{2}{2b+6}}\right)\left(1+Z(t)^{\frac{3}{8}}\right)\\
&&\leq\epsilon Y(t)+C\left(\epsilon\right)\left(1+Z(t)^{\frac{3(b+3)}{8(b+2)}}\right).\nonumber
\end{eqnarray}
As to $I_9^2$, due to
\begin{eqnarray}\label{b115}
I_9^2
&&\leq C\int_{0}^{t}\int_{\mathbb{R}}\theta\left|\left(\frac{\kappa(v,\theta)\theta_{x}}{v}\right)_{x}u_{x}\right|dxds\nonumber\\
&&\leq C\left(\int_{0}^{t}\int_{\mathbb{R}}\frac{u^{2}_{x}}{\theta}dxds\right)^{\frac{1}{2}}
\left(\int_{0}^{t}\int_{\mathbb{R}}(1+\theta)^{3}\left|\left(\frac{\kappa(v,\theta)\theta_{x}}{v}\right)_{x}\right|^{2}dxds\right)^{\frac{1}{2}}\\
&&\leq C\left[\int_{0}^{t}\int_{\mathbb{R}}(1+\theta)^{3}\left(e_{\theta}^{2}(v,\theta)\theta_{t}^{2} +\theta^{2}p_{\theta}^{2}(v,\theta)u_{x}^{2}
+u_{x}^{4}+\phi^{2}z^{2}\right)dxds\right]^{\frac{1}{2}},\nonumber
\end{eqnarray}
and noticing that
\begin{eqnarray}\label{b116}
&&\int_{0}^{t}\int_{\mathbb{R}}(1+\theta)^{3}e_{\theta}^{2}(v,\theta)\theta_{t}^{2}dxds\nonumber\\
&\leq& C\int_{0}^{t}\int_{\mathbb{R}}(1+\theta)^{9}\theta_{t}^{2}dxds\nonumber\\
&\leq& C\left(1+\|\theta\|_{\infty}^{6}\right)X(t)\\
&\leq &C\left(1+Y(t)^{\frac{6}{2b+6}}\right)X(t),\nonumber
\end{eqnarray}
\begin{eqnarray}\label{b117}
&&\int_{0}^{t}\int_{\mathbb{R}}(1+\theta)^{3}\theta^{2}p_{\theta}^{2}(v,\theta)u_{x}^{2}dxds\nonumber\\
&\leq& C\int_{0}^{t}\int_{\mathbb{R}}(1+\theta)^{11}u_{x}^{2}dxds\nonumber\\
&\leq& C\left(1+\|\theta\|_{\infty}^{12}\right)\int_{0}^{t}\int_{\mathbb{R}}\frac{u^{2}_{x}}{\theta}dxds\\
&\leq& C+CY(t)^{\frac{6}{b+3}},\nonumber
\end{eqnarray}
\begin{eqnarray}\label{b118}
&&\int_{0}^{t}\int_{\mathbb{R}}(1+\theta)^{3}u_{x}^{4}dxds\nonumber\\
&\leq&C\left(1+\|\theta\|_{\infty}^{4}\right)\|u_{x}\|^{2}_{\infty}\int_{0}^{t}\int_{\mathbb{R}}\frac{u^{2}_{x}}{\theta}dxds\\
&\leq& C\left(1+Y(t)^{\frac{2}{b+3}}\right)\left(1+Z(t)^{\frac{3}{4}}\right),\nonumber
\end{eqnarray}
and
\begin{eqnarray}\label{bz118}
&&\int_{0}^{t}\int_{\mathbb{R}}(1+\theta)^{3}\phi^{2}z^{2}dxds\nonumber\\
&\leq& C\left(1+\|\theta\|_{\infty}^{\beta+3}\right)\int_{0}^{t}\int_{\mathbb{R}}\phi z^{2}dxds\\
&\leq& C+CY(t)^{\frac{\beta+3}{2b+6}},\nonumber
\end{eqnarray}
we can thus get by combining (\ref{b115})-(\ref{bz118}) that
\begin{eqnarray}\label{b119}
I_9^2&\leq& C\left[1+\left(1+Y(t)^{\frac{6}{2b+6}}\right)X(t)+Y(t)^{\frac{6}{b+3}}+\left(1+Y(t)^{\frac{2}{b+3}}\right) \left(1+Z(t)^{\frac{3}{4}}\right)+Y(t)^{\frac{\beta+3}{2b+6}}\right]^{\frac{1}{2}}\nonumber\\
&\leq& C\left[1+X(t)^{\frac{1}{2}}+X(t)^{\frac{1}{2}}Y(t)^{\frac{3}{2b+6}}+Y(t)^{\frac{3}{b+3}} +Z(t)^{\frac{3}{8}}+Z(t)^{\frac{3}{8}}Y(t)^{\frac{1}{b+3}}+Y(t)^{\frac{\beta+3}{4b+12}}\right]\\
&\leq&\epsilon(X(t)+Y(t))+C\left(\epsilon\right)\left(1+Z(t)^{\frac{3(b+3)}{8(b+2)}}\right).\nonumber
\end{eqnarray}

\eqref{b113} together with \eqref{b114} and \eqref{b119} tell us that
\begin{eqnarray}\label{b120}
I_{9}&&\leq\epsilon(X(t)+Y(t))+C\left(\epsilon\right)\left(1+Z(t)^{\frac{3(b+3)}{8(b+2)}}\right).
\end{eqnarray}

Now for $I_{10}$, we get from \eqref{a4} and the estimate \eqref{b43} obtained in Lemma 2.6 that
\begin{eqnarray}\label{b121}
|I_{10}|&&=\left|\int_{0}^{t}\int_{\mathbb{R}}\frac{\kappa(v,\theta)\theta_{x}}{v}\left(\frac{\kappa(v,\theta)}{v}\right)_{v}v_{x}\theta_{t}dxds\right| \nonumber\\
&&\leq C\int_{0}^{t}\int_{\mathbb{R}}\left(1+\theta^{b}\right)|\theta_{x}v_{x}\theta_{t}|dxds\\
&&\leq\epsilon X(t)+C\left(\epsilon\right)+C\left(\epsilon\right)\int_{0}^{t}
\left\|\frac{\kappa(v,\theta)\theta_{x}}{v}\right\|^{2}_{L^{\infty}\left(\mathbb{R}\right)}\|v_{x}(s)\|^2ds.\nonumber
\end{eqnarray}

For the last term in the right hand side of \eqref{b121}, noticing
\begin{eqnarray}\label{b122}
\max\limits_{x\in\mathbb{R}}\left\{\left(\frac{\kappa(v,\theta)\theta_{x}}{v}\right)^{2}(t,x)\right\}\leq C\int_{\mathbb{R}}\left|\frac{\kappa(v,\theta)\theta_{x}}{v}\right|\left|\left(\frac{\kappa(v,\theta)\theta_{x}}{v}\right)_{x}\right|dx,
\end{eqnarray}
we can get from $\eqref{a1}_{3}$, \eqref{b58}, \eqref{b61} and \eqref{b122} that
\begin{eqnarray}\label{b123}
&&\int_{0}^{t}\left\|\frac{\kappa(v,\theta)\theta_{x}}{v}\right\|^{2}_{L^{\infty}\left(\mathbb{R}\right)}\|v_{x}(s)\|^2ds\nonumber\\
&\leq& C\left(1+\left\|\theta\right\|^{\max\{1,\; (7-b)_{+}\}}_{\infty}\right)\int_{0}^{t}\int_{\mathbb{R}}\left|\frac{\kappa(v,\theta)\theta_{x}}{v}\right| \left|\left(\frac{\kappa(v,\theta)\theta_{x}}{v}\right)_{x}\right|dxds\\
&\leq& C\left(1+\left\|\theta\right\|^{\max\{1,\; (7-b)_{+}\}}_{\infty}\right)\left(\int_{0}^{t}\int_{\mathbb{R}}\frac{\kappa(v,\theta)\theta_{x}^{2}} {v\theta^{2}}dxds\right)^{\frac{1}{2}}
\left(\int_{0}^{t}\int_{\mathbb{R}}\left(1+\theta\right)^{b+2}\left|\left(\frac{\kappa(v,\theta) \theta_{x}}{v}\right)_{x}\right|^{2}dxds\right)^{\frac{1}{2}} \nonumber\\
&\leq& C\left(1+\left\|\theta\right\|^{\max\{1,\; (7-b)_{+}\}}_{\infty}\right)\left(\int_{0}^{t}\int_{\mathbb{R}}\left(1+\theta\right)^{b+2}
\left(e_{\theta}^{2}(v,\theta)\theta_{t}^{2}+\theta^{2}p_{\theta}^{2}(v,\theta)u_{x}^{2}+u_{x}^{4}+\phi^{2}z^{2} \right)dxds\right)^{\frac{1}{2}},\nonumber\\
&\leq& C\left(1+Y(t)^{\frac{\max\{1,\; (7-b)_{+}\}}{2b+6}}\right) \left(\int_{0}^{t}\int_{\mathbb{R}}\left(1+\theta\right)^{b+2}
\left(e_{\theta}^{2}(v,\theta)\theta_{t}^{2}+\theta^{2}p_{\theta}^{2}(v,\theta)u_{x}^{2} +u_{x}^{4}+\phi^{2}z^{2}\right)dxds\right)^{\frac{1}{2}}.\nonumber
\end{eqnarray}

For the last term in the right hand side of \eqref{b123}, one can deduce from $\eqref{a1}_{3}$, \eqref{a2}, \eqref{a3}, \eqref{a4}, \eqref{b58}, \eqref{b61} and \eqref{b122} again that
\begin{eqnarray}\label{b124}
&&\int_{0}^{t}\int_{\mathbb{R}}(1+\theta)^{b+2}e_{\theta}^{2}(v,\theta)\theta_{t}^{2}dxds\nonumber\\
&\leq& C\int_{0}^{t}\int_{\mathbb{R}}(1+\theta)^{b+8}\theta_{t}^{2}dxds\\
&\leq& C\left(1+\|\theta\|_{\infty}^{5}\right)X(t)\nonumber\\
&\leq& C\left(1+Y(t)^{\frac{5}{2b+6}}\right)X(t),\nonumber
\end{eqnarray}
\begin{eqnarray}\label{b125}
&&\int_{0}^{t}\int_{\mathbb{R}}(1+\theta)^{b+2}\theta^{2}p_{\theta}^{2}(v,\theta)u_{x}^{2}dxds\nonumber\\
&\leq& C\int_{0}^{t}\int_{\mathbb{R}}(1+\theta)^{b+10}u_{x}^{2}dxds\\
&\leq& C\left(1+\|\theta\|_{\infty}^{b+11}\right)\int_{0}^{t}\int_{\mathbb{R}}\frac{u^{2}_{x}}{\theta}dxds\nonumber\\
&\leq& C+CY(t)^{\frac{b+11}{2b+6}},\nonumber
\end{eqnarray}
\begin{eqnarray}\label{b126}
&&\int_{0}^{t}\int_{\mathbb{R}}(1+\theta)^{b+2}u_{x}^{4}dxds\nonumber\\
&\leq& C\left(1+\|\theta\|_{\infty}^{b+3}\right)\|u_{x}\|^{2}_{\infty}\int_{0}^{t}\int_{\mathbb{R}}\frac{u^{2}_{x}}{\theta}dxds\\
&\leq& C\left(1+Y(t)^{\frac{1}{2}}\right)\left(1+Z(t)^{\frac{3}{4}}\right),\nonumber
\end{eqnarray}
and
\begin{eqnarray}\label{b127}
&&\int_{0}^{t}\int_{\mathbb{R}}(1+\theta)^{b+2}\phi^{2}z^{2}dxds\nonumber\\
&\leq& C\left(1+\|\theta\|_{\infty}^{b+2+\beta}\right)\int_{0}^{t}\int_{\mathbb{R}}\phi z^{2}dxds\\
&\leq& C+CY(t)^{\frac{b+2+\beta}{2b+6}}.\nonumber
\end{eqnarray}
Consequently, we obtain by combining the estimates (\ref{b123})-(\ref{b127}) that
\begin{eqnarray}\label{b128}
&&\int_{0}^{t}\left\|\frac{\kappa(v,\theta)\theta_{x}}{v}\right\|^{2}_{L^{\infty}\left(\mathbb{R}\right)}\|v_{x}(s)\|^2ds\nonumber\\
&\leq& C\left(1+Y(t)^{\frac{\max\{1,\; (7-b)_{+}\}}{2b+6}}\right)\left[\left(1+Y(t)^{\frac{5}{2b+6}}\right)X(t)\right.\nonumber\\
&&\left.+Y(t)^{\frac{b+11}{2b+6}}+\left(1+Y(t)^{\frac{1}{2}}\right)\left(1+Z(t)^{\frac{3}{4}}\right) +Y(t)^{\frac{b+2+\beta}{2b+6}}\right]^{\frac{1}{2}}\nonumber\\
&\leq& C\left(1+X(t)^{\frac{1}{2}}+X(t)^{\frac{1}{2}}Y(t)^{\frac{5}{4b+12}}+Y(t)^{\frac{b+11}{4b+12}}\right.\nonumber\\
&&+Y(t)^{\frac{1}{4}}+Z(t)^{\frac{3}{8}}+Y(t)^{\frac{1}{4}}Z(t)^{\frac{3}{8}}
+Y(t)^{\frac{b+2+\beta}{4b+12}}+Y(t)^{\frac{\max\{1,\; (7-b)_{+}\}}{2b+6}}\\
&&+X(t)^{\frac{1}{2}}Y(t)^{\frac{\max\{1,\; (7-b)_{+}\}}{2b+6}}+X(t)^{\frac{1}{2}}Y(t)^{\frac{5+2\max\{1,\; (7-b)_{+}\}}{4b+12}}\nonumber\\
&&+Y(t)^{\frac{b+11+2\max\{1,\; (7-b)_{+}\}}{4b+12}}+Y(t)^{\frac{b+3+2\max\{1,\; (7-b)_{+}\}}{4b+12}}+Y(t)^{\frac{\max\{1,\; (7-b)_{+}\}}{2b+6}}Z(t)^{\frac{3}{8}}\nonumber\\
&&\left.+Y(t)^{\frac{b+3+2\max\{1,\; (7-b)_{+}\}}{4b+12}}Z(t)^{\frac{3}{8}}+Y(t)^{\frac{b+\beta+2+2\max\{1,\; (7-b)_{+}\}}{4b+12}}\right)\nonumber\\
&\leq&\epsilon(X(t)+Y(t))+C\left(\epsilon\right)\left(1+Z(t)^{\max\left\{\frac{1}{2}, \frac{3(b+3)}{2(3b+9-2\max\{1,\; (7-b)_{+}\})}\right\}}\right),\nonumber
\end{eqnarray}
where we have used the fact that
\begin{eqnarray*}
Y(t)^{\frac{1}{4}}Z(t)^{\frac{3}{8}}&&\leq\epsilon Y(t)+C\left(\epsilon\right)Z(t)^{\frac{1}{2}},\\
X(t)^{\frac{1}{2}}Y(t)^{\frac{5+2\max\{1,\; (7-b)_{+}\}}{4b+12}}&&\leq\epsilon X(t)+C\left(\epsilon\right)Y(t)^{\frac{5+2\max\{1,\; (7-b)_{+}\}}{2b+6}},\\
&&\leq \epsilon (X(t)+Y(t))+C\left(\epsilon\right),\quad \left(b>\frac{13}{4}\right)\\
Y(t)^{\frac{b+3+2\max\{1,\; (7-b)_{+}\}}{4b+12}}Z(t)^{\frac{3}{8}}&&\leq\epsilon Y(t)+C\left(\epsilon\right)Z(t)^{\frac{3(b+3)}{2(3b+9-2\max\{1,\; (7-b)_{+}\})}},\\
Y(t)^{\frac{b+\beta+2+2\max\{1,\; (7-b)_{+}\}}{4b+12}}&&\leq \epsilon Y(t)+C\left(\epsilon\right).
\end{eqnarray*}

Combining (\ref{b121})-(\ref{b128}), we can get the following estimate on $I_{10}$
\begin{eqnarray}\label{b130}
|I_{10}|&&\leq\epsilon\left(X(t)+Y(t)\right)+C\left(\epsilon\right)\bigg(1+Z(t)^{\max\left\{\frac{1}{2}, \frac{3(b+3)}{2(3b+9-2\max\{1,\; (7-b)_{+}\})}\right\}}\bigg).\nonumber
\end{eqnarray}

Finally, from \eqref{a4}, the estimate \eqref{b43} obtained in Lemma 2.6, the estimates \eqref{b2}, \eqref{b16} and the assumption $0\leq\beta<b+9$, the terms $I_{11}$ and $I_{12}$ can be bounded as follows:
\begin{eqnarray}\label{b135}
|I_{11}|&&=\left|\int_{0}^{t}\int_{\mathbb{R}}\lambda\phi zK_{v}(v,\theta)u_{x}dxds\right| \nonumber\\
&&\leq C\left(1+\|\theta\|_{\infty}^{\frac{\beta+3}{2}}\right)\left(\int_{0}^{t}\int_{\mathbb{R}}\phi z^{2}dxds\right)^{\frac{1}{2}}\left(\int_{0}^{t}\int_{\mathbb{R}}\frac{u^{2}_{x}}{\theta}dxds\right)^{\frac{1}{2}}\\
&&\leq C+CY(t)^{\frac{\beta+3}{4b+12}}\nonumber\\
&&\leq\epsilon Y(t)+C\left(\epsilon\right)\nonumber
\end{eqnarray}
and
\begin{eqnarray}\label{b136}
|I_{12}|&&=\left|\int_{0}^{t}\int_{\mathbb{R}}\frac{\lambda\phi z\kappa\left(v,\theta\right)\theta_{t}}{v}dxds\right| \nonumber\\
&&\leq \epsilon X(t)+C\left(\epsilon\right)\int_{0}^{t}\int_{\mathbb{R}}\left(1+\theta^{b+\beta-3}\right)\phi z^{2}dxds\\
&&\leq \epsilon X(t)+C\left(\epsilon\right)\left(1+\|\theta\|^{b+\beta-3}_{\infty}\right)\int_{0}^{t}\int_{\mathbb{R}}\phi z^{2}dxds\nonumber\\
&&\leq \epsilon (X(t)+Y(t))+C\left(\epsilon\right).\nonumber
\end{eqnarray}

With the above estimates in hand, if we define $\lambda_1$ as in \eqref{bz136},
then combining all the above estimates and by choosing $\epsilon>0$ small enough, we can get the estimate \eqref{b62} immediately. This completes the proof of our lemma.
\end{proof}

Our next result in this section is to show that $Z(t)$ can be bounded by $X(t)$ and $Y(t)$.
\begin{lemma} Under the assumptions listed in Lemma 2.1, we have for all $0\leq t\leq T$ that
\begin{equation}\label{b137}
Z(t)\leq C\left(1+X(t)+Y(t)+Z(t)^{\lambda_{2}}\right).
\end{equation}
Here $\lambda_{2}$ is given by
\begin{equation}\label{b150}
\lambda_{_{2}}=\max\left\{\frac{3b+9}{4b+10}, \frac{3(b+3)}{2(2b+6-\max\{1,\; (7-b)_{+}\})}\right\}.
\end{equation}
It is easy to see that $\lambda_2\in(0,1)$ provided that $b>\frac{11}{3}.$
\end{lemma}
\begin{proof}
Differentiating $(\ref{a1})_{2}$ with respect to $t$, multiplying it by $u_{t}$, then integrating the resulting equation with respect to $t$ and $x$ over $\left(0,t\right)\times\mathbb{R}$, we have
\begin{equation}\label{b138}
\frac{\|u_{t}(t)\|^{2}}{2}+\int_{0}^{t}\int_{\mathbb{R}}\frac{\mu u_{tx}^{2}}{v}dxds=\int_{\mathbb{R}}\frac{u_{0t}^{2}}{2}dx
+\int_{0}^{t}\int_{\mathbb{R}}\bigg(u_{tx}p_{t}(v,\theta)+\frac{\mu u_{x}^{2}u_{tx}}{v^{2}}\bigg)dxds.
\end{equation}

Since
\begin{equation}\label{bz138}
p_{t}(v,\theta)=\left(\frac{R}{v}+\frac{4}{3}a\theta^{3}\right)\theta_{t}-\frac{R\theta u_{x}}{v^{2}},
\end{equation}
we can deduce from the identity (\ref{bz138}) and the estimate \eqref{b43} obtained in Lemma 2.6 that
\begin{eqnarray}\label{b139}
&&\|u_{t}(t)\|^{2}+\int_{0}^{t}\left\|u_{xt}(s)\right\|^2ds\nonumber\\
&\leq& C\left(1+\int_{0}^{t}\int_{\mathbb{R}}\left( p_{t}^{2}(v,\theta)+u_{x}^{4}\right)dxds\right) \nonumber\\
&\leq& C+C\int_{0}^{t}\int_{\mathbb{R}}\left(\left(1+\theta^{6}\right)\theta_{t}^{2}+\theta^{2}u_{x}^{2}\right)dxds +C\int_{0}^{t}\int_{\mathbb{R}}u^{4}_{x}dxds\\
&\leq& C+CX(t)+C\left(\|\theta\|^{3}_{\infty}+\|\theta u^{2}_{x}\|_{\infty}\right)\int_{0}^{t}\int_{\mathbb{R}}\frac{u^{2}_{x}}{\theta}dxds\nonumber\\
&\leq& C\left(1+X(t)+Y(t)+Z(t)^{\frac{3b+9}{4b+10}}\right),\nonumber
\end{eqnarray}
where we have used (\ref{b3}), (\ref{b58}) and (\ref{b61}) again.

Moreover, we can conclude from $(\ref{a1})_{2}$ that
\begin{equation}\label{b140}
u_{xx}= \frac{v}{\mu}\left[u_{t}+p_{x}(v,\theta)+\frac{\mu u_{x}v_{x}}{v^{2}}\right].
\end{equation}

Noticing
\begin{equation}\label{b141}
 p_{x}(v,\theta)=\left(\frac{R\theta}{v}+\frac{a\theta^{4}}{3}\right)_{x}=\frac{R\theta_{x}}{v}-\frac{R\theta v_{x}}{v^{2}}
 +\frac{4}{3}a\theta^{3}\theta_{x},
\end{equation}
one thus gets by combining (\ref{b140}) and (\ref{b141}) that
\begin{eqnarray}\label{b142}
\|u_{xx}(t)\|^{2}&&\leq C\left(1+\int_{\mathbb{R}}\left( u_{t}^{2}+p_{x}^{2}(v,\theta)+u_{x}^{2}v_{x}^{2}\right)dx\right)\nonumber\\
&&\leq C\left(\|u_{t}(t)\|^{2}+\int_{\mathbb{R}}\left(1+\theta^{6}\right)\theta_{x}^{2}dx +\int_{\mathbb{R}}\left(\theta^{2}+u_{x}^{2}\right)v_{x}^{2}dx\right)\\
&&\leq C\left(1+X(t)+Y(t)+Z(t)^{\frac{3b+9}{4b+10}}+\|\theta\|_{\infty}^{2}\|v_{x}(t)\|^{2} +\|u_{x}\|^{2}_{\infty}\|v_{x}(t)\|^{2}\right).\nonumber
\end{eqnarray}

Noticing that
\begin{eqnarray*}
Y(t)^{\frac{\max\{1,\; (7-b)_{+}\}}{2b+6}}Z(t)^{\frac{3}{4}}\leq CY(t)+CZ(t)^{\frac{3(b+3)}{2(2b+6-\max\{1,\; (7-b)_{+}\})}},
\end{eqnarray*}
the last two terms in the right hand side of \eqref{b142} can be bounded as follows:
\begin{eqnarray}\label{b143}
\|\theta\|_{\infty}^{2}\|v_{x}(t)\|^{2}&&\leq C+C\|\theta\|_{\infty}^{2+\max\{1,\; (7-b)_{+}\}}\nonumber\\
&&\leq C+CY(t)^{\frac{2+\max\{1,\; (7-b)_{+}\}}{2b+6}}\\
&&\leq C+CY(t)\nonumber
\end{eqnarray}
and
\begin{eqnarray}\label{b144}
\|u_{x}\|^{2}_{\infty}\|v_{x}(t)\|^{2}&&\leq C+C\bigg(1+\|\theta\|_{\infty}^{\max\{1,\; (7-b)_{+}\}}\bigg)\bigg(1+Z(t)^{\frac{3}{4}}\bigg)\nonumber\\
&&\leq C\bigg(1+Z(t)^{\frac{3}{4}}+Y(t)^{\frac{\max\{1,\; (7-b)_{+}\}}{2b+6}}+Y(t)^{\frac{\max\{1,\; (7-b)_{+}\}}{2b+6}}Z(t)^{\frac{3}{4}}\bigg)\\
&&\leq C+CY(t)+CZ(t)^{\frac{3(b+3)}{2(2b+6-\max\{1,\; (7-b)_{+}\})}}.\nonumber
\end{eqnarray}

Combining (\ref{b142})-(\ref{b144}), we arrive at
\begin{equation}\label{b145}
\|u_{xx}(t)\|^{2}\leq C\left(1+X(t)+Y(t)+Z(t)^{\lambda_{_{2}}}\right),
\end{equation}
where $\lambda_2$ is given by \eqref{b150}.

Having obtained the estimate \eqref{b145}, we can get the estimate \eqref{b137} by using the definition of $Z(t)$, thus the proof of Lemma 2.9 is complete.
\end{proof}

With the above lemmas in hand, we can deduce the uniform upper bound of $\theta\left(t,x\right)$ now. In fact, we have the following lemma.

\begin{lemma} Under the assumptions listed in Lemma 2.1, there exists a positive constant $\overline{\Theta}$ which depends only on the initial data $(v_0(x), u_0(x), \theta_0(x), z_0(x))$, such that
  \begin{equation}\label{b151}
 \theta\left(t,x\right) \leq\overline{\Theta},\quad \forall
 \left(t,x\right)\in[0,T] \times\mathbb{R}.
  \end{equation}
Moreover, we have for $0\leq t\leq T$ that
\begin{equation}\label{b152}
 \left\|\left(v-1, u, \theta-1, z, v_{x}, u_{t}, \theta_{x}, u_{xx}\right)\left(t\right)\right\|^{2}+\int_{0}^{t}\left\|\left(\sqrt{\theta}v_{x}, u_{x}, \theta_{x}, \theta_{t}, u_{xt}, z_{x}\right)(s)\right\|^{2}ds\leq C
\end{equation}
and
\begin{equation}\label{b152-2}
 \int_{0}^{t}\left\|u_{x}(s)\right\|^4_{L^4(\mathbb{R})}ds\leq C,\quad \|u_{x}\|_{L^\infty([0,T]\times\mathbb{R})}\leq C.
\end{equation}
Recall that the constants $C$ in \eqref{b151}, \eqref{b152} and \eqref{b152-2} depend only on the initial data $(v_0(x),$ $u_0(x), \theta_0(x), z_0(x))$.
\end{lemma}
\begin{proof}
Putting (\ref{b137}) and (\ref{b139}) together and using Lemma 2.8, one has
 \begin{eqnarray*}
 W(t)+Z(t)&&\leq C\left(1+X(t)+Y(t)+Z(t)^{\lambda_{2}}\right)\nonumber\\
 &&\leq C\left(1+\epsilon W(t)+Z(t)^{\lambda_{1}}+Z(t)^{\lambda_{2}}\right).
  \end{eqnarray*}

Noticing $0<\lambda_{1}, \lambda_{2}<1$, we get by choosing $\epsilon>0$ small enough and by using Young's inequality that
\begin{equation}\label{b153}
 W(t)+Z(t)\leq C.
\end{equation}

Combining (\ref{b58}), (\ref{b62}) and (\ref{b153}), we can obtain (\ref{b151}) immediately. (\ref{b152}) follows from Lemma 2.1-Lemma 2.9 and (\ref{b151}).
\end{proof}

The next lemma is concerned with the estimate on $\|u_{x}(t)\|$.

\begin{lemma} Under the assumptions listed in Lemma 2.1, we have for any $0\leq t\leq T$ that
\begin{equation}\label{b154}
\|u_{x}(t)\|^{2}+\int_{0}^{t}\left\|u_{xx}(s)\right\|^2ds\leq C.
\end{equation}
\end{lemma}

\begin{proof}
Multiplying $(\ref{a1})_{2}$ by $u_{xx}$ and integrating the result identity with respect to $t$ and $x$ over $\left(0,t\right)\times\mathbb{R}$, one has
\begin{eqnarray}\label{b155}
&&\|u_{x}(t)\|^{2}+\int_{0}^{t}\int_{\mathbb{R}}\frac{\mu u^{2}_{xx}}{v}dxds\\
&\leq& C+C\int_{0}^{t}\int_{\mathbb{R}}\left(\left|\left(\frac{\theta}{v}\right)_{x}u_{xx}\right|
+\theta^{3}\left|\theta_{x}u_{xx}\right|+\frac{\left|u_{x}u_{xx}v_{x}\right|}{v}\right)dxds.\nonumber
\end{eqnarray}
The terms in the right-hand side of (\ref{b155}) can be estimated as follows:
\begin{eqnarray}\label{b156}
&&\int_{0}^{t}\int_{\mathbb{R}}\left|\left(\frac{\theta}{v}\right)_{x}u_{xx}\right|dxds\nonumber\\
&\leq& \int_{0}^{t}\int_{\mathbb{R}}\left(\frac{\left|\theta_{x}u_{xx}\right|}{v}
+\frac{\theta\left|v_{x}u_{xx}\right|}{v^{2}}\right)dxds\\
&\leq& \epsilon\int_{0}^{t}\int_{\mathbb{R}}\frac{\mu u^{2}_{xx}}{v}dxds+C\left(\epsilon\right)\int_{0}^{t}\int_{\mathbb{R}} \left(\frac{\kappa(v,\theta)\theta^{2}_{x}}{v\theta^{2}}\cdot\frac{\theta^{2}}{1+v\theta^{b}}
+\theta^{2}v_{x}^{2}\right)dxds\nonumber\\
&\leq& \epsilon\int_{0}^{t}\int_{\mathbb{R}}\frac{\mu u^{2}_{xx}}{v}dxds+C\left(\epsilon\right),\nonumber
\end{eqnarray}
\begin{eqnarray}\label{b157}
&&\int_{0}^{t}\int_{\mathbb{R}}\theta^{3}\left|\theta_{x}u_{xx}\right|dxds\nonumber\\
&\leq& \epsilon\int_{0}^{t}\int_{\mathbb{R}}\frac{\mu u^{2}_{xx}}{v}dxds+C\left(\epsilon\right)\int_{0}^{t}\int_{\mathbb{R}}\frac{\kappa(v,\theta)\theta^{2}_{x}}{v\theta^{2}}\cdot\frac{\theta^{8}}{1+v\theta^{b}}dxds\\
&\leq& \epsilon\int_{0}^{t}\int_{\mathbb{R}}\frac{\mu u^{2}_{xx}}{v}dxds+C\left(\epsilon\right)\nonumber
\end{eqnarray}
and
\begin{eqnarray}\label{b158}
&&\int_{0}^{t}\int_{\mathbb{R}}\frac{\left|u_{x}u_{xx}v_{x}\right|}{v}dxds\nonumber\\
&\leq& \epsilon\int_{0}^{t}\int_{\mathbb{R}}\frac{\mu u^{2}_{xx}}{v}dxds+C\left(\epsilon\right)\int_{0}^{t}\int_{\mathbb{R}}u_{x}^{2}v_{x}^{2}dxds\nonumber\\
&\leq& \epsilon\int_{0}^{t}\int_{\mathbb{R}}\frac{\mu u^{2}_{xx}}{v}dxds+C\left(\epsilon\right)\int_{0}^{t}\|u_{x}\|^{2}_{L^{\infty}\left(\Omega\right)}ds\\
&\leq& \epsilon\int_{0}^{t}\int_{\mathbb{R}}\frac{\mu u^{2}_{xx}}{v}dxds+C\left(\epsilon\right)\int_{0}^{t}\|u_{x}(s)\|\|u_{xx}(s)\|ds\nonumber\\
&\leq& 2\epsilon\int_{0}^{t}\int_{\mathbb{R}}\frac{\mu u^{2}_{xx}}{v}dxds+C\left(\epsilon\right)\int_{0}^{t}\int_{\mathbb{R}}\frac{\mu u_{x}^{2}}{v\theta}\cdot\theta dxds\nonumber\\
&\leq& 2\epsilon\int_{0}^{t}\int_{\mathbb{R}}\frac{\mu u^{2}_{xx}}{v}dxds+C\left(\epsilon\right),\nonumber
\end{eqnarray}
where we have used Lemma 2.10, Sobolev's inequality and the assumption $b>\frac{11}{3}$.

Combining all the above estimates and by choosing $\epsilon>0$ small enough, we can complete the proof of our lemma.
\end{proof}

The main purpose of the following lemma is to deduce a nice bound on $\int_{0}^{t}\left\|\theta_{xx}(s)\right\|^2ds$.
\begin{lemma} Under the assumptions listed in Lemma 2.1, we have for $0\leq t\leq T$ that
\begin{equation}\label{b160}
\left\|\theta_{x}(t)\right\|^{2}+\int_{0}^{t}\left\|\theta_{xx}(s)\right\|^2ds\leq C.
\end{equation}
\end{lemma}
 \begin{proof}
 We first rewrite $(\ref{a1})_{3}$ in the following form
\begin{equation}\label{b161}
e_{\theta}(v,\theta)\theta_{t}+\theta p_{\theta}(v,\theta)u_{x}=\frac{\mu u_{x}^{2}}{v}+\left(\frac{\kappa(v,\theta)\theta_{x}}{v}\right)_{x}+\lambda\phi z.
  \end{equation}

Multiplying (\ref{b161}) by $-\frac{\theta_{xx}}{e_{\theta}}$ and integrating the result identity with respect to $x$ over $\mathbb{R}$, one has
\begin{eqnarray}\label{b162}
&&\frac{1}{2}\frac{d}{dt}\|\theta_{x}(t)\|^{2}+\int_{\mathbb{R}}\frac{\kappa(v,\theta)\theta^{2}_{xx}}{ve_{\theta}(v,\theta)}dx\nonumber\\
&=&\int_{\mathbb{R}}\left\{\theta p_{\theta}(v,\theta)u_{x}-\frac{\mu u_{x}^{2}}{v}
-\frac{\kappa_{v}(v,\theta)v_{x}+\kappa_{\theta}(v,\theta)\theta_{x}}{v}\theta_{x} +\frac{\kappa(v,\theta)\theta_{x}v_{x}}{v^{2}}-\lambda\phi z\right\}\frac{\theta_{xx}}{e_{\theta}(v,\theta)}dx\nonumber\\
&\leq& \underbrace{\left|\int_{\mathbb{R}}\frac{\theta p_{\theta}(v,\theta)u_{x}\theta_{xx}}{e_{\theta}(v,\theta)}dx\right|}_{I_{13}}
+\underbrace{\left|\int_{\mathbb{R}}\frac{\mu u^{2}_{x}\theta_{xx}}{ve_{\theta}(v,\theta)}dx\right|}_{I_{14}}
+\underbrace{\left|\int_{\mathbb{R}}\frac{\kappa_{v}(v,\theta)v_{x}\theta_{x}\theta_{xx}}{ve_{\theta}(v,\theta)}dx\right|}_{I_{15}}\\
&&+\underbrace{\left|\int_{\mathbb{R}}\frac{\kappa_{\theta}(v,\theta)\theta^{2}_{x}\theta_{xx}}{ve_{\theta}(v,\theta)}dx\right|}_{I_{16}}
+\underbrace{\left|\int_{\mathbb{R}}\frac{\kappa(v,\theta)\theta_{x}v_{x}\theta_{xx}}{v^{2}e_{\theta}(v,\theta)}dx\right| }_{I_{17}}
+\underbrace{\left|\int_{\mathbb{R}}\frac{\lambda\phi z\theta_{xx}}{e_{\theta}(v,\theta)}dx\right|}_{I_{18}}.\nonumber
\end{eqnarray}

The terms $I_k (13\leq k\leq 18)$ can be estimated term by term by employing Cauchy's inequality, Sobolev's inequality and \eqref{b152} as follows:
\begin{eqnarray}\label{b163}
I_{13}
&&\leq\epsilon\int_{\mathbb{R}}\frac{\kappa(v,\theta)\theta^{2}_{xx}}{ve_{\theta}(v,\theta)}dx +C\left(\epsilon\right)\int_{\mathbb{R}}\frac{\mu u^{2}_{x}}{v\theta}\cdot\frac{v^{2}\theta^{3}p^{2}_{\theta}(v,\theta)} {\kappa(v,\theta) e_{\theta}(v,\theta)}dx\nonumber\\
&&\leq\epsilon\int_{\mathbb{R}}\frac{\kappa(v,\theta)\theta^{2}_{xx}}{ve_{\theta}(v,\theta)}dx+C\left(\epsilon\right)V\left(t\right),
\end{eqnarray}
\begin{eqnarray}\label{b164}
I_{14}
&&\leq\epsilon\int_{\mathbb{R}}\frac{\kappa(v,\theta)\theta^{2}_{xx}}{ve_{\theta}(v,\theta)}dx +C\left(\epsilon\right)\int_{\mathbb{R}}\frac{u^{4}_{x}}{\kappa(v,\theta) ve_{\theta}(v,\theta)}dx\nonumber\\
&&\leq\epsilon\int_{\mathbb{R}}\frac{\kappa(v,\theta)\theta^{2}_{xx}}{ve_{\theta}(v,\theta)}dx +C\left(\epsilon\right)\int_{\mathbb{R}}u^{4}_{x}dx,
\end{eqnarray}

\begin{eqnarray}\label{b165}
I_{15}
&&\leq\epsilon\int_{\mathbb{R}}\frac{\kappa(v,\theta)\theta^{2}_{xx}}{ve_{\theta}(v,\theta)}dx
+C\left(\epsilon\right)\int_{\mathbb{R}}v^{2}_{x}\theta^{2}_{x}\cdot\frac{\kappa_{v}^{2}(v,\theta)}{ve_{\theta}(v,\theta) \kappa(v,\theta)}dx\nonumber\\
&&\leq\epsilon\int_{\mathbb{R}}\frac{\kappa(v,\theta)\theta^{2}_{xx}}{ve_{\theta}(v,\theta)}dx
+C\left(\epsilon\right)\left\|\theta_{x}(t)\right\|\left\|\theta_{xx}(t)\right\|\\
&&\leq 2\epsilon\int_{\mathbb{R}}\frac{\kappa(v,\theta)\theta^{2}_{xx}}{ve_{\theta}(v,\theta)}dx +C\left(\epsilon\right)V\left(t\right),\nonumber
\end{eqnarray}
\begin{eqnarray}\label{b168}
I_{16}
&&\leq\epsilon\int_{\mathbb{R}}\frac{\kappa(v,\theta)\theta^{2}_{xx}}{ve_{\theta}(v,\theta)}dx +C\left(\epsilon\right)\int_{\mathbb{R}}\frac{\theta^{4}_{x}\kappa^{2}_{\theta}(v,\theta)}{v\kappa(v,\theta) e_{\theta}(v,\theta)}dx \nonumber\\
&&\leq\epsilon\int_{\mathbb{R}}\frac{\kappa(v,\theta)\theta^{2}_{xx}}{ve_{\theta}(v,\theta)}dx
+C\left(\epsilon\right)\left\|\theta_{x}(t)\right\|\left\|\theta_{xx}(t)\right\|\\
&&\leq2\epsilon\int_{\mathbb{R}}\frac{\kappa(v,\theta)\theta^{2}_{xx}}{ve_{\theta}(v,\theta)}dx +C\left(\epsilon\right)V\left(t\right),\nonumber
\end{eqnarray}
\begin{eqnarray}\label{b169}
I_{17}
&&\leq\epsilon\int_{\mathbb{R}}\frac{\kappa(v,\theta)\theta^{2}_{xx}}{ve_{\theta}(v,\theta)}dx
+C\left(\epsilon\right)\int_{\mathbb{R}}\theta^{2}_{x}v^{2}_{x}\cdot\frac{\kappa(v,\theta)}{v^{3}e_{\theta}(v,\theta)}dx\nonumber\\
&&\leq\epsilon\int_{\mathbb{R}}\frac{\kappa(v,\theta)\theta^{2}_{xx}}{ve_{\theta}(v,\theta)}dx
+C\left(\epsilon\right)\|\theta_{x}(t)\|\|\theta_{xx}(t)\|\\
&&\leq2\epsilon\int_{\mathbb{R}}\frac{\kappa(v,\theta)\theta^{2}_{xx}}{ve_{\theta}(v,\theta)}dx +C\left(\epsilon\right)V\left(t\right),\nonumber
\end{eqnarray}
and
\begin{eqnarray}\label{b170}
I_{18}
&&\leq\epsilon\int_{\mathbb{R}}\frac{\kappa(v,\theta)\theta^{2}_{xx}}{ve_{\theta}(v,\theta)}dx
+C\left(\epsilon\right)\int_{\mathbb{R}}\frac{v\phi^{2}z^{2}}{\kappa(v,\theta) e_{\theta}(v,\theta)}dx\nonumber\\
&&\leq\epsilon\int_{\mathbb{R}}\frac{\kappa(v,\theta)\theta^{2}_{xx}}{ve_{\theta}(v,\theta)}dx
+C\left(\epsilon\right)\int_{\mathbb{R}}\phi z^{2}dx.
\end{eqnarray}
Here we have used the following estimate in deducing the estimates \eqref{b165}, \eqref{b168} and \eqref{b169}
\begin{eqnarray*}
&&\left\|\theta_{x}(t)\right\|\left\|\theta_{xx}(t)\right\|\\
&\leq& \epsilon\int_{\mathbb{R}}\frac{\kappa(v,\theta)\theta^{2}_{xx}}{ve_{\theta}(v,\theta)}dx+
\left(\int_{\mathbb{R}}\frac{\kappa(v,\theta)\theta^{2}_{x}}{v\theta^{2}} \cdot\frac{v\theta^{2}}{\kappa(v,\theta)}dx\right)^{\frac{1}{2}}
\left(\int_{\mathbb{R}}\frac{\kappa(v,\theta)\theta^{2}_{xx}}{ve_{\theta}(v,\theta)} \cdot\frac{ve_{\theta}(v,\theta)}{\kappa(v,\theta)}dx\right)^{\frac{1}{2}}\nonumber\\
&\leq& \epsilon\int_{\mathbb{R}}\frac{\kappa(v,\theta)\theta^{2}_{xx}}{ve_{\theta}(v,\theta)}dx+ CV^{\frac{1}{2}}\left(t\right)\left(\int_{\mathbb{R}}\frac{\kappa(v,\theta)\theta^{2}_{xx}}{ve_{\theta}(v,\theta)}dx\right)^{\frac{1}{2}}\nonumber\\
&\leq& \epsilon\int_{\mathbb{R}}\frac{\kappa(v,\theta)\theta^{2}_{xx}}{ve_{\theta}(v,\theta)}dx+ CV^{\frac{1}{2}}(t).
\end{eqnarray*}

Combining (\ref{b162})-(\ref{b170}) and by choosing $\epsilon>0$ small enough, we arrive at
\begin{equation}\label{b171}
\frac{1}{2}\frac{d}{dt}\|\theta_{x}(t)\|^{2}+\left\|\theta_{xx}(t)\right\|^2\leq C\left(V\left(t\right)+\int_{\mathbb{R}}u^{4}_{x}(t,x)dx+\int_{\mathbb{R}}\phi(t,x) z^{2}(t,x)dx\right).
\end{equation}

Integrating(\ref{b171}) with respect to $t$ over $\left(0, t\right)$ and using \eqref{b152}, we can complete the proof of our lemma.
\end{proof}

The following lemma is concerned with an estimate on $\|z_{x}(t)\|^{2}$.
\begin{lemma} Under the assumptions listed in Lemma 2.1, we can get for any $0\leq t\leq T$ that
\begin{equation}\label{b172}
\left\|z_{x}(t)\right\|^{2}+\int_{0}^{t}\left\|z_{xx}(s)\right\|^2ds\leq C.
\end{equation}
\end{lemma}

\begin{proof}
Multiplying $\eqref{a1}_{4}$ by $z_{xx}$ and integrating the  resulting identity  with respect to $x$ over $\mathbb{R}$, one has
\begin{equation}\label{b174}
\frac{d}{dt}\|z_{x}(t)\|^{2}+\left\|z_{xx}(t)\right\|^2\leq C\int_{\mathbb{R}}\left(\phi z\left|z_{xx}\right|+\frac{\left|z_{x}v_{x}z_{xx}\right|}{v^{3}}\right)(t,x)dx.
\end{equation}

The terms on the right-hand side of (\ref{b174}) can be estimated as follows
\begin{eqnarray}\label{b175}
\int_{\mathbb{R}}\phi z\left|z_{xx}\right|dxds&&\leq\epsilon\int_{\mathbb{R}}z^{2}_{xx}dx
+C\left(\epsilon\right)\int_{\mathbb{R}}\phi^{2} z^{2}dx\nonumber\\
&&\leq\epsilon\int_{\mathbb{R}}z^{2}_{xx}dx
+C\left(\epsilon\right)\|\theta\|^{\beta}_{\infty}\int_{\mathbb{R}}\phi z^{2}dx\\
&&\leq\epsilon\int_{\mathbb{R}}z^{2}_{xx}dx
+C\left(\epsilon\right)\int_{\mathbb{R}}\phi z^{2}dx,\nonumber
\end{eqnarray}
where we have used (\ref{b151}), while
\begin{eqnarray}\label{b176}
\int_{\mathbb{R}}\frac{\left|z_{x}v_{x}z_{xx}\right|}{v^{3}}dx&&\leq\epsilon\int_{\mathbb{R}}z^{2}_{xx}dx
+C\left(\epsilon\right)\int_{\mathbb{R}}z_{x}^{2}v_{x}^{2}dx\nonumber\\
&&\leq\epsilon\int_{\mathbb{R}}z^{2}_{xx}dx
+C\left(\epsilon\right)\|z_{x}\|\|z_{xx}\|\\
&&\leq 2\epsilon\int_{\mathbb{R}}z^{2}_{xx}dx
+C\left(\epsilon\right)\int_{\mathbb{R}}z_{x}^{2}dx,\nonumber
\end{eqnarray}
where we have used \eqref{b152}.

Combining (\ref{b174})-(\ref{b176}) and by choosing $\epsilon>0$ small enough, we can obtain
\begin{equation}\label{b177}
\frac{d}{dt}\|z_{x}(t)\|^{2}+\left\|z_{xx}(t)\right\|^2\leq C\int_{\mathbb{R}}\left(\phi z^{2}+z_{x}^{2}\right)(t,x)dx.
\end{equation}

Then integrating (\ref{b177}) with respect to $t$ over $\left(0, t\right)$ and by using \eqref{b2}, we can obtain \eqref{b172}.
This completes the proof of Lemma 2.13.
\end{proof}

The next lemma is concerned with an estimate on the lower bound of the temperature $\theta\left(t, x\right)$.

\begin{lemma} Under the assumptions stated in Lemma 2.1, for each $0\leq s\leq t\leq T$ and $x\in\mathbb{R}$, the following estimate
  \begin{equation}\label{b180}
  \theta\left(t,x\right)\geq \frac{C\min\limits_{x\in\mathbb{R}}\{\theta(s,x)\}}{1+(t-s)\min\limits_{x\in\mathbb{R}}\{\theta(s,x)\}}
  \end{equation}
holds for some positive constant $C$ which depends only on the initial data $(v_0(x), u_0(x), \theta_0(x), z_0(x))$.
\end{lemma}
\begin{proof}
If we set $h(t, x)=\frac{1}{\theta(t, x)}$, then one can deduce from $(\ref{a1})_{3}$ that
\begin{eqnarray}\label{b181}
 &&e_{\theta}(v,\theta)h_{t}-\left(\frac{\kappa\left(v,\theta\right)h_{x}}{v}\right)_{x}\\
 &=&\frac{vp_{\theta}^{2}(v,\theta)}{4\mu} -\left[\frac{2\kappa\left(v,\theta\right)h_{x}^{2}}{vh}+
  \frac{\mu h^{2}}{v}\left(u_{x}-\frac{vp_{\theta}(v,\theta)}{2\mu h}\right)^{2}+\lambda h^{2}\phi z\right],\nonumber
\end{eqnarray}
where
\begin{equation}\label{b182}
 e_{\theta}(v,\theta)=C_{v}+4av\theta^{3},\quad p_{\theta}(v,\theta)=\frac{R}{v}+\frac{4}{3}a\theta^{3}.
\end{equation}

It is easy to see that
\begin{equation}\label{b183}
 e_{\theta}(v,\theta)h_{t}-\left(\frac{\kappa\left(v,\theta\right)h_{x}}{v}\right)_{x}\leq\frac{vp_{\theta}^{2}(v,\theta)}{4\mu}.
\end{equation}

Using the identities (\ref{b182}), one has from the estimate \eqref{b43} obtained in Lemma 2.6 that
\begin{eqnarray}\label{b184}
h_{t}&\leq& \frac{1}{ e_{\theta}(v,\theta)}\left(\frac{\kappa\left(v,\theta\right)h_{x}}{v}\right)_{x}+C\left(1+\theta^{3}\right)\\
&\leq& \frac{1}{ e_{\theta}(v,\theta)}\left(\frac{\kappa\left(v,\theta\right)h_{x}}{v}\right)_{x} +C\frac{d}{dt}\int_{s}^{t}\max\limits_{x\in\mathbb{R}}\left\{1+\theta^{3}(\tau,x)\right\}d\tau.\nonumber
\end{eqnarray}

If we set
\begin{eqnarray*}
g(t, x)&=&h(t, x)-C\int_{s}^{t}\max\limits_{x\in\mathbb{R}}\left\{1+\theta^{3}(\tau,x)\right\}d\tau\\
&=&\frac{1}{\theta(t,x)}-C\int_{s}^{t}\max\limits_{x\in\mathbb{R}}\left\{1+\theta^{3}(\tau,x)\right\}d\tau,
  \end{eqnarray*}
then we can conclude from (\ref{b184}) that
  \begin{eqnarray*}
g_{t}-\frac{1}{ e_{\theta}(v,\theta)}\left(\frac{\kappa\left(v,\theta\right)g_{x}}{v}\right)_{x}&&\leq 0,\quad 0\leq s<t\leq T,\ x\in\mathbb{R},\\
g(t,x)|_{t=s}&&=h(s,x)=\frac{1}{\theta(s,x)},\quad x\in\mathbb{R}.
  \end{eqnarray*}

With the help of the maximum principle of parabolic equation, we can deduce that
 \begin{eqnarray*}
g(t,x)\leq \max\limits_{x\in\mathbb{R}}\{h(s,x)\}=\frac{1}{\min\limits_{x\in\mathbb{R}}\{\theta(s,x)\}}
\end{eqnarray*}
holds for all $(t,x)\in [s,T]\times\mathbb{R}$.

Hence, we have for all $(t,x)\in [s,T]\times\mathbb{R}$ that
\begin{equation}\label{b189}
h(t,x)\leq \frac{1}{\min\limits_{x\in\mathbb{R}}\{\theta(s,x)\}}+C\int_{s}^{t}\max\limits_{x\in\mathbb{R}} \left\{1+\theta^{3}(\tau,x)\right\}d\tau.
\end{equation}

On the other hand, the estimate \eqref{b151} obtained in Lemma 2.10 tells us that
\begin{equation}\label{b193}
\int_{s}^{t}\max\limits_{x\in\mathbb{R}}\left\{1+\theta^3\left(\tau,x\right)\right\}d\tau\leq C(t-s)
\end{equation}
holds for all $0\leq s\leq t\leq T$.

Having obtained (\ref{b189}) and (\ref{b193}), the estimate \eqref{b180} follows immediately from the definition of $g(t,x)$. This completes the proof of Lemma 2.14.
\end{proof}

\section{Proof of Theorem \ref{Th1.1}}
With Lemma 2.1-Lemma 2.14 in hand, we can turn to prove our main result Theorem \ref{Th1.1}. To this end, we collect the a priori estimates obtained in Lemma 2.1-Lemma 2.14 as follows: For some positive constants $0<M_1\leq M_2, 0<N_1\leq N_2$, suppose that $(v(t,x), u(t,x),\theta(t,x),z(t,x))\in X(0,T;M_1,M_2;N_1,N_2)$ is a solution to the Cauchy problem \eqref{a1}-\eqref{a4}, \eqref{a5}, \eqref{a6} defined on the strip $\Pi_T=[0,T]\times\mathbb{R}$ for some $T>0$, then there exist some positive constants $\underline{V}$, $\overline{V}$, $\overline{\Theta}$, $C_1$ and $C_2$, which depend only on the initial data $(v_0(x), u_0(x), \theta_0(x), z_0(x))$ but are independent of $M_1, M_2, N_1$ and $N_2$, such that the following estimates
\begin{eqnarray}
 \theta\left(t,x\right) &\leq&\overline{\Theta},\label{c-1}\\
 \underline{V}\leq v(t,x)&\leq&\overline{V},\label{c-2}
\end{eqnarray}
\begin{equation}\label{c-3}
 \left\|\left(v-1, u, \theta-1, z\right)(t)\right\|^{2}_{H^1(\mathbb{R})}+\int_{0}^{t}\left(\left\|\sqrt{\theta(s)}v_{x}(s)\right\|^{2}+\left\|\left(u_{x}, \theta_{x}, z_{x}\right)(s)\right\|^{2}_{H^{1}\left(\mathbb{R}\right)}\right)ds\leq C_1
\end{equation}
and
\begin{equation}\label{c-4}
  \theta\left(t,x\right)\geq \frac{C_2\min\limits_{x\in\mathbb{R}}\{\theta(s,x)\}}{1+(t-s)\min\limits_{x\in\mathbb{R}}\{\theta(s,x)\}}
\end{equation}
hold for all $0\leq s\leq t\leq T$ and $x\in\mathbb{R}$.

Now we turn to prove Theorem \ref{Th1.1}. Before doing so, we first point out that the analysis here is different from that of \cite{Li-Liang-ARMA-2016} for one-dimensional viscous heat-conducting ideal polytropic gas motion. In that case since the corresponding global solvability result is well-established in \cite{Antontsev-Kazhikov-Monakhov-1990, Kazhikhov-Shelukhin-JAMM-1977}, the a priori estimates similar to \eqref{c-1}-\eqref{c-4} indeed hold for all $t\in\mathbb{R}^+$ and consequently the stability analysis together with the local lower bound estimate similar to that of \eqref{c-4} can lead to a uniform positive lower bound estimate on $\theta(t,x)$ from which the time-asymptotical behavior of the global solution constructed in  \cite{Antontsev-Kazhikov-Monakhov-1990, Kazhikhov-Shelukhin-JAMM-1977} can be obtained. But for the problem considered in this paper, no global solvability result is available up to now and we had to deal with the following problems:
\begin{itemize}
\item how to extend the local solutions $(v(t,x), u(t,x),\theta(t,x),z(t,x))$ to the Cauchy problem \eqref{a1}-\eqref{a4}, \eqref{a5}, \eqref{a6} step by step to a global one by combining the a priori estimates \eqref{c-1}-\eqref{c-4} with the well-established local solvability result?
\item how to deduce the uniform lower positive bound for $\theta(t,x)$?
\end{itemize}
The key point of our analysis is to introduce a well-designed continuation argument which are motivated by the work of \cite{Wang-Zhao-M3AS-2016} for one-dimensional viscous heat-conducting ideal polytropic gas motion with temperature and density dependent viscosity and can be divided into the following three steps:\vspace{3mm}

\noindent{\bf Step 1:}\quad Let $T_1=24C_1^2$ with $C_1$ being the constant in \eqref{c-3},  by employing the standard continuation argument, one can easily deduce by combining the well-known local solvability result with the a priori estimates \eqref{c-1}-\eqref{c-4} that the local solution $(v(t,x), u(t,x),\theta(t,x),z(t,x))$ to the Cauchy problem \eqref{a1}-\eqref{a4}, \eqref{a5}, \eqref{a6} can be extended to the time interval $[0,2T_1]$. Moreover, the a priori estimates \eqref{c-1}-\eqref{c-4} tell us that the estimates \eqref{c-1}-\eqref{c-3} hold for all $(t,x)\in[0,2T_1]\times\mathbb{R}$ and the estimate \eqref{c-4} with $s=0, T=2T_1$ tells us that
\begin{equation}\label{c-5}
\theta\left(t,x\right)\geq \frac{C_2\min\limits_{x\in\mathbb{R}}\{\theta_0(x)\}}{1+2T_1\min\limits_{x\in\mathbb{R}}\{\theta_0(x)\}}= \frac{n_0C_2}{1+2T_1n_0}
\end{equation}
is true for all $(t,x)\in[0,2T_1]\times\mathbb{R}$.

The fact that the estimate \eqref{c-3} holds for $t\in[0,2T_1]$ together with the fact $\theta_x(t,x)\in C([0,2T_1],$ $L^2(\mathbb{R}))$ imply that there exists a $t_1\in[0,T_1]$ such that
\begin{equation}\label{c-6}
\left\|\theta_x(T_1+t_1)\right\|\leq \frac{1}{4\sqrt{C_1}},
\end{equation}
since if
$$
\left\|\theta_x(t)\right\|>\frac{1}{4\sqrt{C_1}}
$$
holds for all $t\in[T_1,2T_1]$, then we have
$$
\int_0^{2T_1}\left\|\theta_x(s)\right\|^2ds\geq \int^{2T_1}_{T_1}\left\|\theta_x(s)\right\|^2ds>\frac{T_1}{16C_1}=\frac 32C_1>C_1,
$$
which contradicts the fact that the estimate \eqref{c-3} holds for all $t\in[0,2T_2]$.

\eqref{c-6}, Sobolev's inequality and the estimate \eqref{c-3} with $T=2T_1$ imply
\begin{eqnarray*}
\left\|\theta(T_1+t_1)-1\right\|_{L^\infty(\mathbb{R})}&\leq& \left\|\theta(T_1+t_1)-1\right\|^{\frac 12}\left\|\theta_x(T_1+t_1)\right\|^{\frac 12}\\
&\leq& C_1^{\frac 14}\sqrt{\frac{1}{4\sqrt{C_1}}}\\
&=&\frac 12
\end{eqnarray*}
and consequently
\begin{equation}\label{c-7}
\frac 12\leq \theta(T_1+t_1,x)\leq \frac 32
\end{equation}
holds for all $x\in\mathbb{R}$.

\eqref{c-6} together with \eqref{c-7} yield
\begin{equation}\label{c-8}
\theta(t,x)\geq \underline{\Theta}_1\equiv\min\left\{\frac 12, \frac{n_0C_2}{1+2T_1n_0}\right\}
\end{equation}
holds for all $(t,x)\in[0,T_1+t_1]\times\mathbb{R}$.\vspace{3mm}

\noindent{\bf Step 2:}\quad Now take $(v(T_1+t_1,x), u(T_1+t_1,x), \theta(T_1+t_1,x), z(T_1+t_1,x))$ as initial data, then  by employing the standard continuation argument again, one can easily deduce by combining the well-known local solvability result with the a priori estimates \eqref{c-1}-\eqref{c-4} that the solution $(v(t,x), u(t,x),$ $\theta(t,x),z(t,x))$ to the Cauchy problem \eqref{a1}-\eqref{a4}, \eqref{a5}, \eqref{a6} defined on the strip $\Pi_{T_1+t_1}=[0,T_1+t_1]\times\mathbb{R}$ can be further extended to the time interval $[0,3T_1+t_1]$ and $(v(t,x), u(t,x),$ $\theta(t,x),z(t,x))$ satisfies the a priori estimates \eqref{c-1}-\eqref{c-3} with $T=3T_1+t_1$ and the estimate \eqref{c-4} with $T=3T_1+t_1, s=T_1+t_1$, that is
\begin{eqnarray}\label{c-9}
\theta(t,x)&\geq& \frac{C_2\min\limits_{x\in\mathbb{R}}\{\theta(T_1+t_1,x)\}}{1+2T_1\min\limits_{x\in\mathbb{R}}\{\theta(T_1+t_1,x)\}}\nonumber\\
&\geq& \frac{C_2}{2(1+T_1)}
\end{eqnarray}
holds for all $(t,x)\in [T_1+t_1,3T_1+t_1]\times\mathbb{R}$. Here we have used the fact that $\min\limits_{x\in\mathbb{R}}\{\theta(T_1+t_1,x)\}=\frac 12$ which follows from \eqref{c-7}.

Similarly, since \eqref{c-3} holds for $T=3T_1+t_1$, one can conclude that there exists a $t_2\in[0,T_1]$ such that
\begin{equation}\label{c-10}
\left\|\theta_x(2T_1+t_1+t_2)\right\|\leq \frac{1}{4\sqrt{C_1}}
\end{equation}
and consequently
\begin{equation}\label{c-11}
\frac 12\leq \theta(2T_1+t_1+t_2,x)\leq \frac 32
\end{equation}
holds for all $x\in\mathbb{R}$.

\eqref{c-9} together with \eqref{c-11} tell us that
\begin{equation}\label{c-12}
\theta(t,x)\geq \underline{\Theta}_2\equiv\min\left\{\frac 12,  \frac{C_2}{2(1+T_1)}\right\}
\end{equation}
holds for all $(t,x)\in[T_1+t_1,2T_1+t_1+t_2]\times\mathbb{R}$.\vspace{3mm}

\noindent{\bf Step 3:}\quad Now take $(v(2T_1+t_1+t_2,x), u(2T_1+t_1+t_2,x), \theta(2T_1+t_1+t_2,x), z(2T_1+t_1+t_2,x))$ as initial data, then repeating the above argument, one can extend the solution $(v(t,x), u(t,x),$ $\theta(t,x),z(t,x))$ to the Cauchy problem \eqref{a1}-\eqref{a4}, \eqref{a5}, \eqref{a6} defined on the strip $\Pi_{2T_1+t_1+t_2}=[0,2T_1+t_1+t_2]\times\mathbb{R}$ once more to the time interval $[0,4T_1+t_1+t_2]$ and $(v(t,x), u(t,x),$ $\theta(t,x),z(t,x))$ satisfies the a priori estimates \eqref{c-1}-\eqref{c-3} with $T=4T_1+t_1+t_2$ and the estimate \eqref{c-4} with $T=4T_1+t_1+t_2, s=2T_1+t_1+t_2$ which implies that
\begin{eqnarray}\label{c-13}
\theta(t,x)&\geq& \frac{C_2\min\limits_{x\in\mathbb{R}}\{\theta(2T_1+t_1+t_2,x)\}}{1+2T_1\min\limits_{x\in\mathbb{R}}\{\theta(2T_1+t_1+t_2,x)\}}\nonumber\\
&\geq& \frac{C_2}{2(1+T_1)}
\end{eqnarray}
holds for all $(t,x)\in [2T_1+t_1+t_2,4T_1+t_1+t_2]\times\mathbb{R}$. Here we have used the fact that $\min\limits_{x\in\mathbb{R}}\{\theta(2T_1+t_1+t_2,x)\}=\frac 12$ which follows from \eqref{c-11}.

Moreover, by employing the same argument, one can find $t_3\in[0,T_1]$ such that
\begin{equation}\label{c-14}
\frac 12\leq \theta(3T_1+t_1+t_2+t_3,x)\leq \frac 32
\end{equation}
holds for all $x\in\mathbb{R}$ and the estimate
\begin{equation}\label{c-15}
\theta(t,x)\geq \underline{\Theta}_2\equiv\min\left\{\frac 12,  \frac{C_2}{2(1+T_1)}\right\}
\end{equation}
holds for all $(t,x)\in[2T_1+t_1+t_2,3T_1+t_1+t_2+t_3]\times\mathbb{R}$. The most important fact here is that $\underline{\Theta}_2$ is a lower bound of $\theta(t,x)$ for both the time interval $[T_1+t_1,2T_1+t_1+t_2]$ and the time interval $[2T_1+t_1+t_2,3T_1+t_1+t_2+t_3]$.

Repeating the above procedure, we can thus extend the solution $(v(t,x), u(t,x),$ $\theta(t,x),z(t,x))$ to the Cauchy problem \eqref{a1}-\eqref{a4}, \eqref{a5}, \eqref{a6} step by step to a global one. Direct by-products of the above procedure are the following uniform positive lower bound on $\theta(t,x)$
\begin{equation}\label{c-16}
\theta(t,x)\geq \underline{\Theta}\equiv\min\left\{\underline{\Theta}_1,\underline{\Theta}_2\right\},\quad \forall(t,x)\in\mathbb{R}^+\times\mathbb{R}
\end{equation}
and the estimate \eqref{c-1}-\eqref{c-3} hold for all $(t,x)\in\mathbb{R}^+\times\mathbb{R}$.

The estimate \eqref{c-3} together with the estimate \eqref{c-16} imply that
\begin{equation}\label{c-17}
 \left\|\left(v-1, u, \theta-1, z\right)(t)\right\|^{2}_{H^1(\mathbb{R})}+\int_{0}^{t}\left(\left\|v_{x}(s)\right\|^{2}+\left\|\left(u_{x}, \theta_{x}, z_{x}\right)(s)\right\|^{2}_{H^{1}\left(\mathbb{R}\right)}\right)ds\leq C_3
\end{equation}
holds for any $t\in\mathbb{R}^+$ and some positive constant $C_3$ which depends only on the initial data $v_0(x), u_0(x),$ $\theta_0(x), z_0(x))$.

Having obtained \eqref{c-17}, the estimate \eqref{a13} can be obtained by the standard method, we thus omit the details for brevity. Thus completes the proof of Theorem \ref{Th1.1}.

\begin{center}
{\bf Acknowledgement}
\end{center}
This work was supported by a grant from the National Natural Science Foundation of China under
contract 11671309 and the Fundamental Research Funds for the Central Universities.


\begin{thebibliography}{99}\small

\bibitem{Antontsev-Kazhikov-Monakhov-1990} S.N. Antontsev, A.V. Kazhikov and V.N. Monakhov, {\it Boundary Value Problems in Mechanics of Nonhomogeneous Fluids.} Amsterdam, New York, 1990.

\bibitem{Chen-SIMA-1992} G.-Q. Chen, Global solutions to the compressible Navier-Stokes equations for a reacting mixture. {\it SIAM J. Math. Anal.} {\bf 23} (1992), no. 3, 609-634.

\bibitem{Chen-Zhao-Zou-PRSE-2017} Q. Chen, H.-J. Zhao and Q.-Y. Zou, Initial-boundary value problems to the one-dimensional compressible Navier-Stokes equations with degenerate transport coefficients. {\it Proc. Roy. Soc. Edinburgh Sect. A}, in press. DOI:10.1017/S0308210516000342.

\bibitem{Dafermos-Hsiao-NonliAnal-1982} C.M. Dafermos and L. Hsiao, Global smooth thermomechanical processes in one-dimensional nonlinear thermoviscoelasticity. {\it Nonlinear Anal.} {\bf 6} (1982), no. 5, 435-454.

\bibitem{Donatelli-Trivisa-CMP-2006} D. Donatelli and K. Trivisa, On the motion of a viscous compressible radiative-reacting gas. {\it Comm. Math. Phys.} {\bf 265} (2006),  no. 2,  463-491.

\bibitem{Ducomet-M3AS-1996} B. Ducomet, On the stability of a stellar structure in one dimension. {\it Math. Models Methods Appl. Sci.} {\bf 6} (1996),  no. 3, 365-383.

\bibitem{Ducomet-MMNA-1997} B. Ducomet, On the stability of a stellar structure in one dimension II: The reactive case. {\it Math. Model. Number. Anal.} {\bf 31} (1997),  no. 3, 381-407.

\bibitem{Ducomet-MMAS-1999} B. Ducomet, A model of thermal dissipation for a one-dimensional viscous reactive and radiative gas. {\it Math. Methods Appl. Sci.} {\bf 22} (1999), no. 15, 1323-1349.

\bibitem{Ducomet-BanachCenterPubl-2000} B. Ducomet, Some stability results for reactive Navier-Stokes-Poisson systems. {\it Evolution equations: existence, regularity and singularities (Warsaw, 1998)},  83-118, {\it Banach Center Publ.}, {\bf 52}, Polish Acad. Sci., Warsaw, 2000.


\bibitem{Ducomet-ARMA-2004} B. Ducomet and E. Feireisl, On the dynamics of gaseous stars. {\it Arch. Ration. Mech. Anal.} {\bf174 } (2004),  no. 2, 221-266.

\bibitem{Ducomet-Feireisl-CMP-2006} B. Ducomet and E. Feireisl, The equations of magnetothydrodynamics: on the interaction between matter and radiation in the evolution of gaseous stars. {\it Comm. math. Phys.} {\bf 266} (2006),  no. 3, 595-626.



\bibitem{Ducomet-Zlotnik-CRASP-2004} B. Ducomet and A. Zlotnik, Stabilization for 1D radiative and reactive viscous gas flow. {\it C. R. Acad. Sci., Paris } {\bf 338} (2004), no. 2, 127-132.

\bibitem{Ducomet-Zlotnik-ARMA-2005} B. Ducomet and A. Zlotnik, Lyapunov functional method for 1D radiative and reactive viscous gas dynamics. {\it Arch. Ration. Mech. Anal.} {\bf 177} (2005),  no. 2, 185-229.

\bibitem{Ducomet-Zlotnik-NonliAnal-2005} B. Ducomet and A. Zlotnik, On the large-time behavior of 1D radiative and reactive viscous flows for higher-order kinetics. {\it Nonlinear Anal.} {\bf 63} (2005), no. 8, 1011-1033.

\bibitem{He-Liao-Wang-Zhao-2017} L. He, Y.-K. Liao, T. Wang and H.-J. Zhao, Work in progress.

\bibitem{Huang-Wang-Xiao-KRM-2016} B.-K. Huang, L.-S. Wang, and Q.-H. Xiao, Global nonlinear stability of rarefaction waves for compressible Navier-Stokes equations with temperature and density dependent transport coefficients. {\it Kinet. Relat. Models} {\bf 9} (2016), no. 3, 469-514.


\bibitem{Jenssen-Karper-SIMA-2010}H.K. Jenssen and T.K. Karper, One-dimensional compressible flow with temperature dependent transport coefficients. {\it SIAM J. Math. Anal.} {\bf  42} (2010), 904-930.

\bibitem{Jiang-ZHeng-JMP-2012} J. Jiang and S.-M. Zheng, Global solvability and asymptotic behavior of a free boundary problem for the one-dimensional viscous radiative and reactive gas. {\it J. Math. Phys.} {\bf 53} (2012), 1-33.

\bibitem{Jiang-ZHeng-ZAMP-2014}  J. Jiang and S.-M. Zheng, Global well-posedness and exponential stability of solutions for the viscous radiative and reactive gas. {\it Z. Angew. Math. Phys.} {\bf 65} (2014), 645-686.

\bibitem{Jiang-AMPA-1998}  S. Jiang, Large-time behavior of solutions to the equations of a viscous  polytropic ideal gas. {\it Ann. Mat. Pura Appl.} {\bf 175} (1998), 253-275.

\bibitem{Jiang-CMP-1999} S. Jiang, Large-time behavior of solutions to the equations of a one-dimensional viscous polytropic ideal gas in unbounded domains. {\it Commun. Math. Phys.} {\bf 200} (1999), 181-193.

\bibitem{Jiang-PRSE-2002} S. Jiang, Remarks on the asymptotic behaviour of solutions to the compressible Navier-Stokes equations in the half-line. {\it Proc. Roy. Soc. Edinb. A} {\bf 132} (2002), 627-638.


\bibitem{Kawohl-JDE-1985} B. Kawohl, Global existence of large solutions to initial boundary value problems for a viscous, heat-conducting, one-dimensional real gas. {\it J.Differential Equations} {\bf 58} (1985), no. 1, 76-103.

\bibitem{Kazhikhov-Shelukhin-JAMM-1977} A.V. Kazhikhov and V.V. Shelukhin, Unique global solution with respect to time of initial boundary value problems for one-dimensinal equations of a viscous gas. {\it J. Appl. Math. Mech.} {\bf 41} (1977), no. 2, 273-282.



\bibitem{Li-Liang-ARMA-2016}J. Li and Z. Liang, Some uniform estimates and large-time behavior of solutions to one-dimensional compressible Navier-Stokes system in unbounded domains with large data. {\it Arch. Ration. Mech. Anal.} {\bf 220} (2016), 1195-1208.

\bibitem{Liao-Zhao} Y.-K. Liao and H.-J. Zhao, Global solutions to one-dimensional equations for a self-gravitating viscous radiative and reactive gas with density-dependent viscosity. {\it Commun. Math. Sci.}, in press.

\bibitem{Liu-Yang-Zhao-Zou-SIMA-2014} H.-X. Liu, T. Yang, H.-J. Zhao and Q.-Y. Zou, One-dimensional compressible Navier-Stokes equations with temperature dependent transport coefficients and large data. {\it SIAM J. Math. Anal.} {\bf 46}  (2014),  no. 3, 2185-2228.


\bibitem{Mihalas-Mihalas-1984} D. Mihalas and B.W. Mihalas, {\it Foundations of Radiation Hydrodynamics.} Oxford Univ. Press, New York, 1984.

\bibitem{Pan-Zhang-CMS-2015}R. Pan and W. Zhang, Compressible Navier-Stokes equations with   temperature dependent heat conductivity. {\it Commun. Math. Sci.} {\bf 13} (2015), 401-425.

\bibitem{Qin-Hu-JMP-2011} Y.-M. Qin and G.-L. Hu, Global smooth solutions for 1D thermally radiative magnetohydrodynamics. {\it J. Math. Phys. } {\bf 52} (2011), no. 2, 023102.


\bibitem{Qin-Hu-Wang-Huang-Ma-JMAA-2013} Y.-M. Qin, G.-L. Hu, T.-G. Wang, L. Huang and Z.-Y. Ma, Remarks on global smooth solutions to a 1D self-gravitating viscous radiative and reactive gas. {\it J. Math. Anal. Appl.} {\bf 408} (2013), no. 1, 19-26.

\bibitem{Qin-Zhang-Su-Cao-JMFM-2016}Y.-M. Qin, J.-L. Zhang, X. Su and J. Cao, Global existence and exponential stability of spherically symmetric solutions to a compressible combustion radiative and reactive gas. {\it J. Math. Fluid Mech.} {\bf 18} (2016), no. 3, 415-461.


\bibitem{Shandarin-Zeldovichi-RMP-1989} S.F. Shandarin and Y.B. Zel'dovich, The large-scale structure of the universe: Turbulence, intermittency, structures in a self-gravitating medium. {\it Rev. Modern Phys.} {\bf 61} (1989), no. 2, 185-220.

\bibitem{Tan-Yang-Zhao-Zou-SIMA-2013} Z. Tan, T. Yang, H.-J. Zhao and Q.-Y. Zou, Global solutions to the one-dimensional compressible Navier-Stokes-Poisson equations with large data. {\it SIAM J. Math. Anal.} {\bf 45} (2013), no. 2, 547-571.


\bibitem{Umehara-Tani-JDE-2007} M. Umehara and A. Tani, Global solution to one-dimensional equations for a self-gravitating viscous radiative and reactive gas. {\it J. Differential Equations} {\bf 234} (2007), no. 2, 439-463.

\bibitem{Umehara-Tani-PJA-2008} M. Umehara and A. Tani, Global solvability of the free-boundary problem for one-dimensinal motion of a self-gravitating viscous radiative and reactive gas. {\it Proc. Japan Acad. Ser. A Math. Sci. } {\bf 84} (2008), no. 7, 123-128.

\bibitem{Wan-Wang-JDE-2017} L. Wan and T. Wang, Symmetric flows for compressible heat-conducting fluids with temperature dependent viscosity coefficients, J. Differential Equations, accepted, DOI: 10.1016/j.jde.2017.02.022.

\bibitem{Wan-Wang-Zhao-JDE-2016} L. Wan, T. Wang, and H.-J. Zhao, Asymptotic stability of wave patterns to compressible viscous and heat-conducting gases in the half space. {\it J. Differential Equations} {\bf 261} (2016), 5949-5991.


\bibitem{Wan-Wang-Zou-Nonlinearity-2016} L. Wan, T. Wang, and Q.-Y. Zou, Stability of stationary solutions to the outflow problem for full compressible Navier-Stokes equations with large initial perturbation. {\it Nonlinearity} {\bf 29} (2016), no. 4, pp. 1329-1354.

\bibitem{Wang-Zhao-M3AS-2016} T. Wang and H.-J. Zhao, One-dimensional compressible heat-conducting gas with temperature-dependent viscosity. {\it Math. Models Methods Appl. Sci.} {\bf 26} (2016), no. 12, 2237-2275.

\bibitem{Wylen-Sonntag-1985} G.J.V. Wylen and R.E. Sonntag, {\it Fundamentals of Classical Thermodynamics.} Wiley, New York, 1985.

\bibitem{Zeldovich-Raizer-1967} Y.B. Zel'dovich and Y.P.  Raizer, {\it Physics of Shock Waves and High-temperature Hydrodynamic Phenomena, vol.  II.} Academic Press, New York, 1967.

\bibitem{Zhang-Xie-JDE-2008} J.-W. Zhang and F. Xie, Global solution for a one-dimensional model problem in thermally radiative magnetohydrodynamics. {\it J. Differential Equations} {\bf 245} (2008), no. 7, 1853-1882.

\end{thebibliography}
\end{document}